\documentclass[10pt,reqno]{amsart}
\usepackage[usenames]{color}
\usepackage{amsmath,verbatim,graphicx,epstopdf,enumerate,amssymb}
\everymath{\displaystyle}
\allowdisplaybreaks

\newcommand{\lb}{\left(}

\newcommand{\rb}{\right)}
\newcommand{\PD}{\partial}
\renewcommand{\d}{\delta}
\newcommand{\Beq}{\begin{equation}}
\newcommand{\Eeq}{\end{equation}}
\newcommand{\beq}{\begin{equation*}}
\newcommand{\eeq}{\end{equation*}}
\newcommand{\bal}{\begin{align}}
\newcommand{\eal}{\end{align}}

\renewcommand{\L}{\langle}

\usepackage{mathtools}

\newcommand{\vp}{\varphi}
\newcommand{\A}{\alpha}
\newcommand{\B}{\beta}
\newcommand{\bp}{\begin{prob}}
	\newcommand{\ep}{\end{prob}}
\newcommand{\bpr}{\begin{proof}}
	\newcommand{\epr}{\end{proof}}

\newcommand{\D}{\mathrm{d}}

\newcommand{\bel}[1]{\begin{equation}\label{#1}}
\newcommand{\ee}{\end{equation}}

\usepackage[colorlinks=true]{hyperref}
\hypersetup{urlcolor=red, citecolor=red, linkcolor= blue}
\usepackage[notref,notcite]{}
\newtheorem{theorem}{Theorem}[section]
\newtheorem{lemma}[theorem]{Lemma}
\newtheorem{proposition}[theorem]{Proposition}
\newtheorem{corollary}[theorem]{Corollary}
\newtheorem{definition}[theorem]{Definition}
\newtheorem{remark}[theorem]{Remark}

\numberwithin{equation}{section}

\usepackage[margin=1in]{geometry}
\newcommand\numberthis{\addtocounter{equation}{1}\tag{\theequation}}
\usepackage{xcolor}

\usepackage{amsmath,verbatim,graphicx,epstopdf,enumerate,mathtools}

\usepackage[centerlast,small,sc]{caption}
\usepackage[usenames]{color}
\usepackage{amsmath,verbatim,graphicx,epstopdf,enumerate}
\renewcommand{\d}{\delta}

\newcommand{\Ac}{\mathcal{A}}

\newcommand{\Dc}{\mathcal{D}}
\newcommand{\Ec}{\mathcal{E}}

\newcommand{\Sc}{\mathcal{S}}
\newcommand{\Tc}{\mathcal{T}}

\newcommand{\Nb}{\mathbb{N}}

\newcommand{\Rb}{\mathbb{R}}
\newcommand{\Sb}{\mathbb{S}}
\newcommand{\Sn}{\mathbb{S}^{n-1}}

\newcommand{\R}{\rangle}
\newcommand{\p}{\partial}
\newcommand{\Dxi}{\D \mathrm{S}_\xi}

\renewcommand{\L}{\langle}

\newcommand{\dbc}[1]{\begin{quotation}\textbf{\color{teal}Divyansh's comment:\
		}{\color{teal}\textit{#1}}\end{quotation}}

\title[UCP for generalized ray transforms]{Unique continuation results for certain generalized ray transforms of symmetric tensor fields}
\author[D.\ Agrawal, V.~ P.~ Krishnan and S. ~K.~ Sahoo]{Divyansh Agrawal$^{\ast}$,  Venkateswaran  P. Krishnan$^{\mathsection}$ and Suman Kumar Sahoo$^{\dagger}$}
\address {$^{\ast}$ Centre for Applicable Mathematics, Tata Institute of Fundamental Research, India.
\newline
E-mail:{\tt\  agrawald@tifrbng.res.in}}
\address{$^{\mathsection}$ Centre for Applicable Mathematics, Tata Institute of Fundamental Research, India.
\newline
E-mail:{\tt \ vkrishnan@tifrbng.res.in}}
\address{$^{\dagger}$ Department of Mathematics and Statistics, University of Jyv\"askyl\"a, Finland.
\newline
E-mail:{\tt \ suman.k.sahoo@jyu.fi}}

\begin{document}
\begin{abstract}
Let $I_{m}$ denote the Euclidean ray transform acting on compactly supported symmetric $m$-tensor field distributions $f$, and $I_{m}^{*}$ be its formal $L^2$ adjoint. We study a unique continuation result for the normal operator  $N_{m}=I_{m}^{*}I_{m}$. More precisely, we show that if $N_{m}$ vanishes to infinite order at a point $x_0$ and if the Saint-Venant operator $W$ acting on $f$ vanishes on an open set containing $x_0$, then $f$ is a potential tensor field. This generalizes two recent works of Ilmavirta and M\"onkk\"onen who proved such unique continuation results for the ray transform of functions and vector fields/1-forms. One of the main contributions of this work is identifying the Saint-Venant operator acting on higher order tensor fields as the right generalization of the exterior derivative operator acting on 1-forms, which makes unique continuation results for ray transforms of higher order tensor fields possible.  In the second  half of the paper, we prove analogous unique continuation results for momentum ray and transverse ray transforms.
% \dbc {Needs modification to include results on MRT, TRT as well.}
% \dbc{the footnotes are just some to-dos}
\end{abstract}

\maketitle

\section{Introduction}

The purpose of this paper is to prove unique continuation properties (UCP) for three Euclidean ray transforms of symmetric $m$-tensor fields; the (usual) ray transform, momentum ray transform and transverse ray transform. Roughly speaking, we show the following: Let $f$ be a compactly supported $m$-tensor field distribution and $U$ be a non-empty open subset of $\Rb^n$ for $n\geq 2$.
\begin{enumerate}
    \item If the ray transform of $f$ vanishes $u$ and if the Saint-Venant operator acting on $f$ vanishes on the same open set, then $f$ is a potential tensor field.
    \item If certain momentum ray transforms of $f$ vanish on a set of lines passing through $U$ and if the generalized Saint-Venant operator acting on $f$ vanishes on the same open set, then $f$ is a generalized potential tensor field.
    \item Let $n\geq 3$. If the transverse ray transform of $f$ vanishes on a set of lines passing through $U$ and if $f$ vanishes on the same open set, then $f$ vanishes identically.
\end{enumerate}

We actually prove stronger versions of some of the the statements mentioned above; see the precise statements of the theorems in the concluding paragraphs of Section \ref{PSM}. % for the precise statements.

The study of the three transforms on symmetric tensor fields is motivated by applications in several applied fields. The investigation of ray transform of symmetric 2-tensor fields is motivated by applications in travel-time tomography \cite{Sharafutdinov_Book,Uhlmann-TTT} and that of symmetric 4-tensor fields in elasticity \cite{Sharafutdinov_Book}. % appears naturally in the mathematical study of seismic imaging \cite{Keijo_partial_function}, medical imaging \cite{.....}, travel-time tomography \cite{Sharafutdinov_Book} and elasticity \cite{Sharafutdinov_Book}. 
The study of momentum ray transforms was introduced by Sharaftudinov \cite{Sharafutdinov_Book} and a more detailed investigation of this transform was undertaken in \cite{Abhishek-Mishra,Krishnan2018,Krishnan2019a,SumanSIAM}. Analysis of such transforms appeared recently in the solution of a Calder\'on-type inverse problem for polyharmonic  operators; see \cite{BKS}.
Transverse ray transform of symmetric tensor fields appear in the study of polarization tomography \cite{Sharafutdinov_Book,Novikov-Sharafutdinov,Lionheart-Sharafutdinov} and X-ray diffraction strain tomography \cite{Lionheart-Withers,Desai-Lionheart}.

We note that the recovery of a symmetric $m$-tensor field $f$ from the knowledge of its  ray transform $I_mf$ is an over-determined problem in dimensions $ n\ge 3$. However, the recovery of $f$ given the normal operator $ N_mf=I_m^*I_mf$, when viewed as a convolution operator; see \eqref{NormalOperator_Ray}, is a formally determined inverse problem, where $I_m^{*}$ is the formal $L^2$ adjoint. Furthermore, we study a partial data problem, that is,  the recovery of $f$, from the knowledge of $N_mf$ and a component of $f$ given in a fixed open subset of $\mathbb{R}^n$.  We prove a unique continuation result for this as well as for two other transforms; the momentum ray and transverse ray transforms. The motivation for a result of this kind for the ray transform comes from its connection to the fractional Laplacian operator. The inversion formula for the recovery of a function (for instance) from its corresponding normal operator is given by the following formula: $f=C (-\Delta)^{1/2}I_0^* I_0f$, where $C$ is a constant that depends only on dimension. Unique continuation results for the fractional operators have a long history and go back to the works of Riesz \cite{Riesz} and Kotake-Narasimhan \cite{Kotake-Narasimhan}. A unique continuation result for a fractional Schr\"odinger equation with rough potentials was done in \cite{Ruland_fractional}. In the context of a Calder\'on-type inverse problem involving the fractional Laplacian, a unique continuation result was employed to prove uniqueness in \cite{GSU-fractional}. %  involving the fractional Laplacian was employed in \cite{GSU-fractional} recently %This adds extra difficulty to this problem and we overcome this difficulty using unique continuation of fractional Laplace operator.

Unique continuation results for the ray transform of functions, the $d$-plane transform and the Radon transform were initiated in the paper \cite{Keijo_partial_function}. Later this was extended to a unique continuation result for the Doppler transform, which deals with the ray transform of a vector field or equivalently a 1-form in \cite{Keijo_partial_vector_field}. One added difficulty in dealing with the Doppler transform or ray transform of higher order symmetric tensor fields is that the ray transform has an infinite dimensional kernel. Therefore unique recovery of the full symmetric tensor field from its ray transform is not possible. Going back to the paper \cite{Keijo_partial_vector_field}, roughly speaking, the main result of the paper reads as follows: Let $f$ be a compactly supported vector field/1-form  and suppose $\D f=0$ on an non-empty open set $U$, where $\D f$ is the exterior derivative of the 1-form $f$, and if the Doppler transform of $f$ vanishes along all lines intersecting $U$, then $\D f\equiv 0$ in $\Rb^n$.   The results of our paper can be viewed as a generalization of this work. The approach of \cite{Keijo_partial_vector_field} is to reduce the unique continuation result for the Doppler transform to that of the scalar ray transform of each component of $\D f$. We follow their idea of reducing to a unique continuation result for a scalar function for the symmetric tensor field case, however, our approach as well as the technique of proof are different. Our main contribution is in identifying the right analogue of the exterior derivative operator to higher order symmetric tensor fields case, which turns out to be the Saint-Venant operator, to prove the unique continuation results for ray transform of higher order symmetric tensor fields. We also prove unique continuation results for momentum ray transforms as well as for transverse ray transform of symmetric tensor fields. To prove unique continuation results for the momentum ray transform, we consider the generalized Saint-Venant operator introduced by Sharafutdinov \cite{Sharafutdinov_Book}. In fact, we define an equivalent version of the generalized Saint-Venant operator from \cite{Sharafutdinov_Book} suitable for our purposes to prove our result.

%We now the give the main results of the article. For the notation used in these statements, we refer the reader to Section \ref{PSM}.

The article is organized as follows. In Section \ref{PSM}, we give the requisite preliminaries and give the statement of the main results. Readers familiar with the integral geometry literature may choose to skip the parts of the section where we fix the notation required to give the statements of the theorems. Instead, they may go directly to the results near the end of Section \ref{PSM} and refer back to the preliminary material as and when required. 
Sections \ref{UCP:Ray}, \ref{UCP:MRT} and \ref{UCP:TRT} give the proofs of the unique continuation results for ray transform, momentum ray transform and transverse ray transform, respectively. 

  %Our purpose for defining them here is to fix the notation. 
%We now define some important operators that will be used throughout this article. Most of these are standard operators and notations, and can be found in \cite{Sharafutdinov_Book}. 

\section{Preliminaries and statements of the main results}\label{PSM}
To state the main results of this work, we begin by defining the operators that will be used throughout the article. Most of these are standard in integral geometry literature and the reference is the book by Sharafutdinov \cite{Sharafutdinov_Book}. For the purpose of fixing the notation, we give them here.
\subsection{Definitions of some operators}
We let $T^m= T^m\Rb^n$ denote the complex vector space of $\Rb$-multi-linear functions from $\underbrace{\Rb^n \times \cdots \times \Rb^n}_{m \text{ times}} \to \mathbb{C}$. Let $e_1,\cdots, e_n$ be the standard basis for $\Rb^n$.  Given an element $u\in T^m$, we let $u_{i_1 \cdots i_m}= u(e_{i_1}, \cdots, e_{i_m})$. These are the components of the tensor $T^m$.  

Given $u\in T^m$ and $v\in T^k$, the tensor product $u\otimes v \in T^{m+k}$ is defined by 
\[
(u\otimes v)(x_1,\cdots, x_m, x_{m+1},\cdots,x_{m+k}) = u(x_1,\cdots, x_m) v(x_{m+1},\cdots, x_{m+k}).
\]

By $S^m = S^m\Rb^n$, we mean the subspace of $T^m$ that are symmetric in all its $m$ arguments. More precisely, $u$ is an element of $S^m$ if 
\[
u_{i_1 \cdots i_m} = u_{i_{\pi(1)} \cdots i_{\pi(m)}}
\] 
for any $\pi \in \Pi_m$ -- the group of permutations of the set $\{ 1,\cdots, m\}$.

Let $\sigma: T^m \to S^m$ be the symmetrization operator defined as follows: 
\[
\sigma u(e_1,\cdots, e_m) = \frac{1}{m!} \sum\limits_{\pi \in \Pi_m} u(e_{\pi(1)},\cdots, e_{\pi(m)}).
\]

The symmetrized tensor product of two tensors will be denoted by $\odot$ instead of $\otimes$. That is, given $u\in T^m$ and $v\in T^k$, 
\[
\lb u\odot v\rb (x_1,\cdots,x_m,x_{m+1},\cdots, x_{m+k})=\frac{1}{(m+k)!}\sum\limits_{\pi\in \Pi_{m+k}} u(x_{\pi(1)},\cdots, x_{\pi(m)})v(x_{i_{\pi(m+1)}},\cdots, x_{i_{\pi(m+k)}}).
\]

Given indices $i_1, \cdots, i_m$ the operator of partial symmetrization with respect to the indices $i_1, \dots ,i_p$, where $p<m$,  of a tensor $u\in T^m $ is given by
\begin{equation*}
	\sigma (i_1 \dots i_p) u_{i_1 \dots i_m} = \frac{1}{p!} \sum\limits_{\pi \in \Pi_p} u_{i_{\pi(1)} \dots i_{\pi(p)} i_{p+1} \dots i_m},
\end{equation*}
where $\Pi_p$ denotes the group of permutations of the set $\{1, \dots , p\}$. 

%Let $U\subset \Rb^n$ be open and consider the disjoint union of tensor fields
%A tensor field or a symmetric tensor field defined on a set $U\subset \Rb^n$ is an function whose value at each point $x\in U$ is a tensor field or a symmetric tensor field.
We next define symmetric tensor fields. If $\Ac\subset \Dc'(\Rb^n)$, the space of $\Ac$-valued symmetric tensor field distributions of $\Rb^n$ is defined by $\Ac(\Rb^n; S^m) = \Ac\otimes_{\mathbb{C}} S^m$. Denote by  $C^{\infty}(\Rb^n; S^m )$, $\Sc(\Rb^n;S^m)$, $C_c^{\infty}(\Rb^n; S^m)$, $\Dc'(\Rb^n; S^m)$, $\Sc'(\Rb^n; S^m)$ and  $\Ec'(\Rb^n; S^m)$ the space of symmetric $m$-tensor fields in $\Rb^n$ whose components are  smooth, Schwartz class, smooth and compactly supported functions, tensor field, tempered tensor field, and  compactly supported tensor field distributions, respectively. An analogous definition is valid when $S^m$ above is replaced by $T^m$. %In other words, 

The family of oriented lines in $\Rb^n$ is parameterized by 
\[
T\Sb^{n-1}=\{(x,\xi)\in \Rb^n\times \Sb^{n-1}: \langle x,\xi\rangle =0\},
\]
where $\langle \cdot, \cdot \rangle$ denotes the standard dot product in $\Rb^n$. 

The alternation operator $\A$ with respect to two indices $i_1, i_2$ is defined by  
\begin{equation*}
	\alpha(i_1 \, i_2) u_{i_1 i_2 j_1 \dots j_p} = \frac{1}{2} (u_{i_1 i_2 j_1 \dots j_p} - u_{i_2 i_1 j_1 \dots j_p}).
\end{equation*}

For $u \in S^k$, we denote by $i_u : S^m  \to S^{m+k}$ the operator of symmetric multiplication by $u$ and by $j_u : S^{m+k} \to S^{m}$ the operator dual to $i_u$. These are given by
\begin{align}
	\label{ioperator} (i_u f)_{i_1 \dots i_{m+k}} &= \sigma (i_1 \dots i_{m+k}) u_{i_1 \dots i_k} f_{i_{k+1} \dots i_{k+m}}    \\
	\label{jop} (j_u g)_{i_1 \dots i_{m}} &= g_{i_1 \dots i_{m+k}} u^{i_{m+1} \dots i_{m+k}}.
\end{align}
For the case in which $u$ is the Euclidean metric tensor, we denote $i_u$ and $j_u$ by $i$ and $j$ respectively. In \eqref{jop} and henceforth, we use the Einstein summation convention, that when the indices are repeated, summation in each of the repeating index varying from $1$ up to the dimension $n$ is assumed.%, for the case when $u$ is the Euclidean metric tensor. %Note that $jf = 0$ for $f \in S^0 $ or $S^1$. 

Next we define two important first order differential operators. The operator of inner differentiation or symmetrized derivative is denoted as $\D : C^{\infty}(\Rb^n; S^m) \to C^{\infty}(\Rb^n; S^{m+1})$ given by 
\begin{equation*}
	(\D f)_{i_1 \dots i_{m+1}} = \sigma(i_1 \dots i_{m+1}) \frac{\p f_{i_1 \dots i_m}}{\p x_{i_{m+1}}}.
\end{equation*}
The divergence operator $\d : C^{\infty}(\Rb^n; S^m) \to C^{\infty}(\Rb^n; S^{m-1})$ is defined by 
\begin{equation}\label{def_of_delta}
	(\d f)_{i_1 \dots i_{m-1}} = \frac{\p f_{i_1 \dots i_m}}{\p x^{i_{m}}}.
\end{equation}
%To avoid any confusion, we stress that the summation is assumed over the repeated index ($i_{m}$ in the above equation). 
The operators $\D$ and $-\d$ are formally dual to each other with respect to $L^2$ inner product.
\[
\langle u, v\rangle = \int u_{i_1 \cdots i_m} \overline{v}^{i_1 \cdots i_m} \, \D x.
\]

Note that the above definitions make sense for compactly supported tensor field distributions as well.

We now recall the  solenoidal-potential decomposition of compactly supported symmetric tensor field distributions \cite{Sharafutdinov_Book}. Let $n\geq 2$. For $f\in \Ec'(\Rb^n; S^m)$, there exist uniquely determined fields
$^s\!f\in  \Sc'(\Rb^n; S^m)$ and $v\in \Sc'(\Rb^n; S^{m-1})$ tending to $0$ at $\infty$ and satisfying
\[
f=\!^s\!f + \D v, \quad \delta ^s\!f=0.
\]
The fields $^s\!f$ and $\D v$ are called the solenoidal and potential components of $f$, respectively. The fields
$^s\! f$ and $v$ are smooth outside supp $f$ and satisfy the estimates
\[
\lvert ^s\!f(x)\rvert \leq C(1+|x|)^{1-n}, \lvert v(x)\rvert \leq C(1+|x|)^{2-n}, \lvert 
\D v(x)\rvert \leq C(1+|x|)^{1-n}.
\]
%The fields $^s\!f$ and $\D v$ belong to $L^2(\Rb^n; S^m)$

%outside some compact

%\subsection{The Saint-Venant operator and its generalizations} In the present section, we recall definition of  the generalized Saint-Venant operator from \cite{Sharafutdinov_Book} which we use to state our main result. %which for $0 \leq k \leq m$ is a differential operator of order $m-k$. We then introduce a natural generalisation of \eqref{RSV}, denoted $R^k$ and proceed to show that these two operators are indeed equivalent.
Finally, we define  the Saint-Venant operator $W$ and the generalized Saint-Venant operator $W^k$ for $0\leq k\leq m$.
%\begin{definition} \cite{Sharafutdinov_Book}

The Saint-Venant operator $W: C^{\infty}(\Rb^n; S^m) \to C^{\infty}(\Rb^n; S^m \otimes S^m )$ is the  differential operator of order $m$ defined by
\begin{equation}\label{SV}
	(Wf)_{i_1 \dots i_m j_1 \dots j_m} = \sigma(i_1 \dots i_m) \sigma(j_1 \dots j_m) \sum\limits_{p=0}^m (-1)^p \binom{m}{p} \frac{\p^m f_{i_1 \dots i_{m-p} j_1 \dots j_p}}{\p x_{j_{p+1}} \dots \p x_{j_m} \p x_{i_{m-p+1}}, \dots \p x_{i_m}},
\end{equation}
where $S^m \otimes S^m$ denotes the set of tensors symmetric with respect to the group of first and last $m$ indices.
%\end{definition}

For our purposes, we will be using an equivalent formulation of the Saint-Venant operator using the operator $R: C^{\infty}(\Rb^n; S^m) \to C^{\infty}(\Rb^n; T^{2m})$ defined as follows:  %Before we do so, we define the operator $R$ as $R: \Ec(S^m; \Rb^n) \to \Ec(T^{2m}; \Rb^n)$
For $f \in C^{\infty}(\Rb^n; S^m)$,
\begin{equation}\label{RSV}
	(Rf)_{i_1 j_1 \dots i_m j_m} = \alpha(i_1 j_1) \dots \alpha(i_m j_m) \frac{\p^m f_{i_1 \dots i_m}}{\p x_{j_1} \dots \p x_{j_m}}.
\end{equation}
The  tensor field $R$ is skew-symmetric with respect to each pair of indices $(i_1, j_1), \dots (i_m, j_m)$, and symmetric with respect to these pairs. %It is known that the operators $R$ and $W$ are equivalent, i.e.,  $Wf = 0 \iff Rf = 0$, which holds even locally. In fact, we have the explicit formulas (see  \cite[Equations 2.4.6 and 2.4.7]{Sharafutdinov_Book})
The operators $R$ and $W$ are equivalent. More precisely, 
\begin{align}
	(Wf)_{i_1 \dots i_m j_1 \dots j_m} &= 2^m\sigma(i_1 \dots i_m) \sigma(j_1 \dots j_m) (Rf)_{i_1 j_1 \dots i_m j_m}\label{RtoW}
	\\ (Rf)_{i_1 j_1 \dots i_m j_m} &= \frac{1}{(m+1)} \alpha(i_1 j_1) \dots \alpha(i_m j_m) (Wf)_{i_1 \dots i_m j_1 \dots j_m}\label{WtoR}.
\end{align}
%Note that for $m=1$, both $W$ and $R$ have the same expression, and coincide with external differentiation.
%\dbc{possible error in \ref{WtoR}. I think it should be $\frac{1}{m+1}$.}
We remark that in \cite{Sharafutdinov_Book}, these formulas have two minor typos;  the factor $2^m $ is missing in \eqref{RtoW} and the constant in \eqref{WtoR} is incorrectly written as $(m+1)$ on the right hand side.
%\dbc{should previous line be removed?}

%We adopt \cite{SumanSIAM}. 
We also need the following generalization \cite{Sharafutdinov_Book} of the Saint-Venant operator. %  for our unique continuation results for MRTs.

For $m \geq 0$ and $0 \leq k \leq m$, the generalized Saint-Venant operator $W^k : C^{\infty}(\Rb^n; S^m) \to C^{\infty}(\Rb^n; S^{m-k} \otimes S^m )$ is defined as
\begin{align}\label{GSV}
	(W^k f)_{p_1 \dots p_{m-k} q_1 \dots q_{m-k} i_1 \dots i_k} =& \sigma(p_1, \dots , p_{m-k}) \sigma(q_1, \dots , q_{m-k}, i_1 \dots i_k) 
	\\ & \sum_{l=0}^{m-k} (-1)^l \binom{m-k}{l} \frac{\p^{m-k} f^{i_1 \dots i_k}_{p_1 \dots p_{m-k-l} q_1 \dots q_{l}}}{\p x^{p_{m-k-l+1}} \dots \p x^{p_{m-k}} \p x^{q_{l+1}} \dots \p x^{q_{m-k}}}, \notag
\end{align}
where we adopt the notation from \cite{SumanSIAM} and by $f^{i_1 \dots i_k}_{p_1 \dots p_{m-k}}$, we mean a tensor field of order $m-k$ with the indices on the top fixed ($i_1 \dots i_k$ here). Note that for $k=0$, \eqref{GSV} agrees with \eqref{SV}, and for $k=m$, $W^m = I$ -- the identity operator. 

Next we  define a generalization of the operator $R$ as follows. For $f\in C^{\infty}(\Rb^n; S^m )$,  %defined for $f \in \Ec(S^m; \Rb^n)$ by the equality
\begin{equation}
	\begin{aligned}\label{an_alternate_definition}
		(R^k f)_{p_1 q_1 \dots p_{m-k} q_{m-k} i_1 \dots i_k} &= \alpha (p_1 q_1) \dots \alpha (p_{m-k} q_{m-k}) \frac{\p^{m-k} f^{i_1 \dots i_k}_{p_1 \dots p_{m-k}}}{\p x_{q_1} \dots \p x_{q_{m-k}}}
		\\ &= R(f^{i_1 \dots i_k})_{p_1 q_1 \dots p_{m-k} q_{m-k}}.
	\end{aligned}
\end{equation}
As with the equivalence of $W$ and $R$, we show the equivalence of $W^k$ and $R^k$ in Section \ref{UCP:MRT}.
%where the operator $R^0$ in the last expression is the operator $R$ defined in \eqref{RSV} acting on the tensor field of order $m-k$.

Finally, note that the operators $W, R, W^k$ and $R^k$ are well-defined for tensor field distributions as well.

\begin{comment}
We also recall that the \emph{tensor product} $f \otimes g \in T^{m+k}(\Rb^n)$ is defined for $f \in T^m(\Rb^n) , g \in T^k(\Rb^n) $ by the equality 
\begin{equation*}
(f \otimes g )_{i_1 \dots i_{m+k}} = f_{i_1 \dots i_m} g_{i_{m+1} \dots i_{m+k}}
\end{equation*}

The space of oriented lines in $\Rb^n$ is parameterized by points on the tangent bundle of the unit sphere $\Sn$ and is defined by
\begin{equation*}
T\Sn = \{(x,\xi) \in \Rb^n \times \Rb^n \, | \,\, |\xi|=1,  \L x, \xi \R = 0\}
\end{equation*}
We will need the spaces $\Dc(S^m; T\Sn), \Ec(S^m; T\Sn)$ and their duals, which denote the corresponding spaces on the manifold $T\Sn$ defined similar to the Euclidean setting.  
\begin{definition}\label{def:ray_transform}\cite{Sharafutdinov_Book}
The ray transform acting on symmetric tensor fields  $I_m: \Dc(S^m; \Rb^n) \to \Dc(S^0; T\Sn)$ is a continuous linear map  given by
\begin{equation}
I_mf(x,\xi) = \int\limits_{-\infty}^\infty \L f(x+t\xi), \xi^m \R \, dt
\end{equation}
\end{definition}
\begin{definition}\label{def_mrt}\cite{Sharafutdinov_Book,Krishnan2019}
For all non-negative integers $k$, the momentum ray transform $I_m^k$ 
\begin{equation*}
I_m^k : \Dc(S^m; \Rb^n) \to \Dc(T\mathbb{S}^{n-1})
\end{equation*}
is defined as
\begin{equation}
(I_m^kf)(x,\xi) = \int\limits_{-\infty}^{\infty} t^k \langle f(x+t\xi),\xi^m \rangle \, dt \hspace{3mm} \text{for} \hspace{3mm} (x,\xi) \in T\mathbb{S}^{n-1}
\end{equation}
\end{definition}
\noindent Note that for $k=0$, $I_m^k$ coincides with the operator of ray transform $I_m$.
\end{comment}
\subsection{Ray, momentum and transverse ray transforms}
We now define the ray transform, momentum and transverse ray transforms whose unique continuation properties we study in this paper. We initially define these transforms on the space of smooth compactly supported symmetric tensor fields. These will be extended to compactly supported tensor field distributions later.

The ray transform $I$ is the bounded linear operator 
\[
I_m: C_c^{\infty}(\Rb^n; S^{m}) \to C_c^{\infty}(T\Sb^{n-1})
\]
defined as follows: 
\Beq\label{def:ray_transform}
I_mf (x,\xi) = \int\limits_{\Rb} f_{i_1 \cdots i_m}(x+t\xi) \xi^{i_1} \cdots \xi^{i_m} \, \D t=\int\limits_{\Rb} \langle f(x+t\xi), \xi^{ m}\rangle \, \D t
\Eeq
This can be naturally extended to points $(x,\xi)\in \Rb^n \times \Rb^n\setminus \{0\}$ using the same definition. We will denote the extended operator by $J_m: C_c^{\infty}(\Rb^n; S^m) \to C^{\infty}(\Rb^n \times \Rb^n\setminus \{0\})$: 
\[
J_mf(x,\xi)= \int f_{i_1 \cdots i_m}(x+t\xi) \xi^{i_1} \cdots \xi^{i_m} \,\D t.
\]
In fact, the operators $I_m$ and $J_m$ are equivalent. Restricting $J_mf$ to points on $T\Sb^{n-1}$ determines $I_mf$. For the other way, we use the following homogeneity properties: 
\begin{align*}
	&J_m(x,r\xi) =  \frac{r^{m}}{|r|} (J_m^k f)(x, \xi) \quad \mbox{for} \quad r\neq 0,\\
	&J_m(x+s\xi,\xi)= J_m(x,\xi).
	\end{align*}
From this we have 
\begin{align*}
	J_mf(x,\xi)= |\xi|^{m-1}I_m f\lb x-\frac{\langle x,\xi\rangle \xi}{|\xi|^2}, \frac{\xi}{|\xi|}\rb \mbox{ for } (x,\xi)\in \Rb^n\times \Rb^n\setminus \{0\}.
\end{align*}

The main reason for working with $J_m$ instead of $I_m$ is that the partial derivatives $\frac{\PD (Jf)}{\PD x_i}$ and $\frac{\PD (Jf)}{\PD \xi_i}$ are well-defined for all $1\leq i\leq n$.

The momentum ray transforms  are the bounded linear operators, $I_m^k: C_c^{\infty}(\Rb^n; S^m) \to C_c^{\infty}(T\Sb^{n-1})$, defined for each $k\geq 0$ as 
\Beq\label{def_mrt}
(I_m^kf)(x,\xi) = \int\limits_{-\infty}^{\infty} t^k f_{i_1\cdots i_m}(x+t\xi)\xi^{i_1} \cdots \xi^{i_m}  \, \D t.
\Eeq
The operator $I_m^0$ is of course the ray transform $I$ defined above.

Similar to the case of ray transform, the momentum ray transforms can be extended for points $(x,\xi)\in \Rb^n \times \Rb^n \setminus \{0\}$. The extended operators will be denoted by
\begin{equation}
	J_m^k : C_c^{\infty}(\Rb^n; S^m) \to C^{\infty}(\mathbb{R}^n \times \mathbb{R}^n\backslash \{0\})
\end{equation}
using the same definition. 
%defined as 
%\begin{equation*}
%	(J_m^kf)(x,\xi) = \int\limits_{-\infty}^{\infty} t^k  f_{i_1 \cdots i_m}(x+t\xi),\xi^{i_1} \cdots \xi^{i_m}\, \D t.
%\end{equation*}
The operators $J_m^k$  satisfy the following \cite{Krishnan2018}: %for $0 \neq r \in \Rb$,
\begin{align}
	(J_m^k f) (x, r \xi) &= \frac{r^{m-k}}{|r|} (J_m^k f)(x, \xi) \quad \mbox{for} \quad r\neq 0 \label{homo_mrt1}
	\\ (J_m^k f)(x+ s \xi, \xi) &= \sum\limits_{l=0}^k \binom{k}{l} (-s)^{k-l} (J_m^l f) (x, \xi) \quad \mbox{for}\quad s\in \mathbb{R}.
	\label{homo_mrt2}
\end{align}

The data $(I_m^0f, I_m^1f, \dots , I_m^kf)$ and $(J_m^0f, J_m^1f, \dots , J_m^kf)$  for $0\leq k\leq m$  are equivalent (\cite[Equations 2.5, 2.6]{Krishnan2018}). As with the case of the ray transform, it is convenient to work with the operators $J_m^k$ because the partial derivatives $\frac{\partial (J_m^kf)}{\partial x^i}$ and $\frac{\partial (J_m^kf)}{\partial \xi^i}$ are well-defined for all $1 \leq i \leq n$. 
%For $k=0$, the relations between $I_m$ and $J_m$ read:

%\begin{align}
%   (I_m^0 f)(x,\xi) &= (J_m^0 %f)(x,\xi) |_{T\Sn}
%   \\ (J_m^0 f)(x,\xi) &= |\xi|^{m-1} (I_m^0 f)\left(x-\frac{(x \cdot \xi)}{|\xi|^2} \xi , \frac{\xi}{|\xi|}\right) \label{ItoJ}
%\end{align}

For $f \in C_c^{\infty}(\Rb^n; S^m)$, we define the transverse ray transform. Let 
$$ T\Sb^{n-1}  \oplus T\Sb^{n-1}  = \{(\omega, x, y) \in \Sb^{n-1} \times \Rb^n \times \Rb^n: \omega \cdot x =0 ,  \omega \cdot y  =0\}$$
be the Whitney sum.

The transverse ray transform  $\Tc: C_c^\infty(\Rb^n; S^m) \rightarrow C^\infty(T\Sb^{n-1}  \oplus T\Sb^{n-1})$ is the bounded linear map defined by 
\Beq\label{def:trt}
\Tc f(\omega, x,y) =  \int_\Rb f_{i_1 \cdots i_m}(x+t\omega) y^{i_1} \cdots y^{i_m}\, \D t.
\Eeq

%These transforms can be defined on compactly supported tensor field distributions. 

\subsection{Normal operators}

Next we extend the definitions of the ray transforms  to compactly supported tensor field distributions. We also define the corresponding normal operators. %corresponding to the ray transform and the momentum ray transform.

 For the case of ray transform \eqref{def:ray_transform},  the definition can be extended to compactly supported tensor field distributions  as done in \cite{Sharafutdinov_Book}: 

For $f\in \Ec'(\Rb^n; S^m)$,  we define $I_mf\in \Ec'(T\Sb^{n-1})$ as 

\[
\langle I_m f , \vp \rangle = \langle f, I_m^{*} \vp\rangle \mbox{ for } \vp \in C^{\infty}(T\Sb^{n-1}),
\]
where 
\[
\lb I_m^{*}\rb _{i_1 \cdots i_{m}} \vp(x)= \int\limits_{\Sb^{n-1}} \xi_{i_1} \cdots \xi_{i_m}\vp(x-\langle x,\xi\rangle \xi,\xi)\, \Dxi.
\]
Here and henceforth, $\Dxi$ is the Euclidean surface measure on the unit sphere.

Similarly, if we work with $J_m$, we can define $J_m f$ for $f\in \Ec'(\Rb^n; S^m)$ using the following formal $L^2$ adjoint: 
\[
\lb J_m^{*}\rb_{i_1 \cdots i_m} \vp(x)=\int\limits_{\Sb^{n-1}}\int\limits_{\Rb}\xi_{i_1} \cdots \xi_{i_m} \vp(x-t\xi, \xi) \D t \Dxi \mbox{ for } \vp \in C_c^{\infty}(\Rb^n \times \Sb^{n-1}).
\]

The normal operator $N_m f= I_m^{*} I_m f$ has the following integral representation \cite{Sharafutdinov_Book}.
\Beq\label{NormalOperator_Ray}
\lb N_mf\rb_{i_1 \cdots i_m}= 2 f_{j_1 \cdots j_m} * \frac{ \lb x^{\odot 2m}\rb_{j_1 \cdots j_m i_1 \cdots i_m}}{\lVert x  \rVert^{2m+n-1}}.
\Eeq
This representation makes sense for $f\in \Ec'(\Rb^n; S^m)$ as the convolution of a compactly supported distribution and a tempered distribution.

Next we define momentum ray transform of compactly supported tensor distributions. This was studied recently in the context of an inverse problem for polyharmonic operators in \cite{BKS}. We work in a slightly different context here and for this reason, we give the details. 

Let us first derive a representation for the formal $L^2$ adjoint $\lb I_m^k\rb^{*}$ of $I_m^k$. %Let $(I_m^k)^*$ denote the formal $L^2$ adjoint of $I_m^k$. 
%For $\xi \neq 0$, let us denote by $\xi^\perp$ the set of vectors orthogonal to $\xi$, i.e., $\xi^\perp = \{ x \in \Rb^n: (x \cdot \xi) = 0\}$. 
Consider for $f \in C_c^{\infty}(\Rb^n)$ and $g \in C^{\infty}(T\Sn)$
\begin{align*}
	\langle I_m^k f, g \rangle_{T\mathbb{S}^{n-1}}&=\langle f, (I_m^k)^*g \rangle_{\mathbb{R}^n} 
	\\ &= \int\limits_{\mathbb{S}^{n-1}} \int\limits_{\xi^{\perp}} (I_m^kf)(x,\xi)\, g(x,\xi)  \,\D x\,  \Dxi
	\\ &= \int\limits_{\mathbb{S}^{n-1}} \int\limits_{\xi^{\perp}} \int\limits_{-\infty}^{\infty}t^k \, \langle f(x+t\xi),\xi^m \rangle \, \D t \, g(x,\xi) \,\D x \,\Dxi
	\\ &=  \int\limits_{\mathbb{S}^{n-1}} \int\limits_{\mathbb{R}^n} \langle z,\xi\rangle^k \, \langle f(z),\xi^m \rangle \, g(z-\langle z, \xi\rangle\xi,\xi) \,\D z \,\Dxi,
	\intertext{where we employed the change of variables $x + t\xi = z$ for $x \in \xi^{\perp}$ and $t \in \mathbb{R}$ for each fixed $\xi \in \mathbb{S}^{n-1}$. Now  interchanging the order of integration }
	\langle f, (I_m^k)^*g \rangle_{\mathbb{R}^n} &= \int\limits_{\mathbb{R}^n} f_{i_1 \dots i_m}(z) \, \left\{ \int\limits_{\mathbb{S}^{n-1}} \langle z,\xi\rangle^k \xi_{i_1} \dots \xi_{i_m} \, g(z-\langle z,\xi\rangle\xi,\xi) \, \D \xi \right \} \, \D z.
	\intertext{%where summation over repeated indices is assumed.
		Thus the formal $L^2$-adjoint of $I_m^k$,  $(I_m^k)^* : C^{\infty}(T\Sn) \to C^{\infty}(\Rb^n; S^m)$ is given by the expression} 
	(I_m^k)^*g_{i_1 \dots i_m} (x) &= \int\limits_{\mathbb{S}^{n-1}} \langle x, \xi\rangle^k \, \xi_{i_1} \cdots \xi_{i_m} \, g(x-\langle x, \xi\rangle \xi, \xi) \, \Dxi. \numberthis{} \label{adjoint_mrt}
\end{align*}
%Note that it is essential to have $x-(x \cdot \xi)\xi$ in the above expression because $g$ is only defined for points on $T\Sn$.

Using this we can extend momentum ray transforms for compactly supported tensor field distributions as follows. $I_m^k: \Ec'(\Rb^n; S^m) \to \Ec'(T\Sn)$ given by
\begin{equation}
	\L I_m^kf , g \R = \L f , (I_m^k)^* g \R
\end{equation}
for $f \in \Ec'(\Rb^n; S^m)$ and $ g \in C^{\infty}(T\Sn)$. 

Similarly, if we work with $J_m^{k}$, then for $f\in \Ec'(\Rb^n; S^m)$, $J_m^k f$ can be defined as follows: 

\[
	\L J_m^kf , g \R = \L f , (J_m^k)^* g \R
\]
where 
\[
(J_m^k)_{i_1 \cdots i_m}^* g(x)= \int\limits_{\Sb^{n-1}}\int\limits_{\Rb} t^{k}g(x-t\xi, \xi)\xi^{i_1} \cdots \xi^{i_m}\, \D t\, \Dxi \quad \mbox{ for } \quad g\in C_c^{\infty}(\Rb^{n}\times \Sb^{n-1}).
\]

%Similarly, $(I_m^k)^*: \Dc'(S^0; T\Sn) \to \Dc'(S^m; \Rb^n)$ acts on $g \in \Dc'(S^0; T\Sn)$ to give $(I_m^k)^*g \in \Dc'(S^m; \Rb^n)$ given as 
%\begin{equation}
%\L I_m^k f, g \R=	\L f,(I_m^k)^*g  \R \quad \mbox{for} \quad f \in \Ec'(\Rb^n; S^m).
%\end{equation}
%\ssc{I have  doubt about this statement $I_m^k: \Ec'(S^m; \Rb^n) \to \Ec'(S^0; T\Sn)$.}

We next study normal operators of momentum ray transforms. %The normal operator related to ray transform is well studied; see \cite{Sharafutdinov_Book}.  
Let us denote $N_m^k= (I_m^k)^* I_m^k : C_c^{\infty}(\Rb^n; S^m) \to C^{\infty}(\Rb^n; S^m) $ be the normal operator of the $k^{\mathrm{th}}$ momentum ray transform of a symmetric $m$-tensor field. By \eqref{adjoint_mrt},
\begin{align*}
	(N_m^k \, f)_{i_1 \dots i_m} (x) &= (I_m^k)_{i_1 \cdots i_m}^* \, I_m^k \, f \, (x)
	\\ &= \int\limits_{\mathbb{S}^{n-1}} \langle x, \xi\rangle^k \, \xi^{i_1} \dots \xi^{i_m} \, I^k f(x-\langle x, \xi\rangle \xi, \xi) \, \Dxi.
	%\\ &= \int\limits_{\mathbb{S}^{n-1}} \int\limits_{-\infty}^\infty (x \cdot \xi)^k \, \xi_{i_1} \dots \xi_{i_m} \, t^k \, f_{j_1 \dots j_m}(x-(x \cdot \xi)\xi + t\xi) \xi_{j_1} \dots \xi_{j_m} \, dt \, d\xi
	%\intertext{Now, breaking the inner integral into two parts, for $t > 0$ and $t < 0$ respectively, we perform the change of variables in the second integral $t \to -t$ and then $\xi \to -\xi$, which entails}
\end{align*}

Note that for $x \in \mathbb{R}^n$ and $\xi \in \mathbb{S}^{n-1}$, $(x- \langle x, \xi\rangle \xi,\xi) \in T\mathbb{S}^{n-1}$. Since $I_m^k$ and $J_m^k$ agree on $T\Sn$, we have % (equation $2.5$ in \cite{Krishnan2018}), 
\begin{align}\label{Normal}
	(N_m^k \, f)_{i_1 \dots i_m} (x) &= \int\limits_{\mathbb{S}^{n-1}} \langle x, \xi\rangle^k \, \xi_{i_1} \dots \xi_{i_m} \, J_m^k f(x-\langle x, \xi\rangle \xi, \xi) \, \Dxi.
\end{align}

%Before we extend the definition of the Normal operator to include field-distributions, we would like to simplify the above expression, and put it in a form which would make sense even for field-distributions.
%Let us simplify \eqref{Normal}.
Using \eqref{homo_mrt2}, we have
\begin{align*}
	(N_m^k \, f)_{i_1 \dots i_m} (x) &= \sum\limits_{l=0}^k \binom{k}{l} \int\limits_{\Sn} \langle  x, \xi\rangle^{2k-l} \xi_{i_1} \dots \xi_{i_m} (J_m^l f) (x, \xi)\, \Dxi
	\\ &= \sum\limits_{l=0}^k \binom{k}{l} \int\limits_{\Sn} \int\limits_{-\infty}^\infty  \langle x, \xi\rangle^{2k-l} \, \xi_{i_1} \dots \xi_{i_m} \,  t^l \, f_{j_1 \dots j_m}(x + t \xi) \xi^{j_1} \dots \xi^{j_m}\, \D t\,  \Dxi 
	\\ &= 2 \sum\limits_{l=0}^k \binom{k}{l} \int\limits_{\Sn} \int\limits_{0}^\infty  \langle x, \xi\rangle^{2k-l} \, \xi_{i_1} \dots \xi_{i_m} \,  t^l \, f_{j_1 \dots j_m}(x + t \xi) \xi^{j_1} \dots \xi^{j_m} \,\D t\, \Dxi.
\end{align*}
Consider the change of variable $x + t\xi = y$ to obtain $t = |y-x|$, $\xi = \frac{y-x}{|y-x|}$. We have

\begin{equation*}
\begin{aligned}
    &(N_m^k f)_{i_1 \dots i_m} (x)\\& \quad= 2 \sum\limits_{l=0}^k \binom{k}{l} \int\limits_{\Rb^n} \langle x, \frac{y-x}{|y-x|} \rangle^{2k-l} \lb \left(\frac{y-x}{|y-x|}\right)^{\odot 2m}\rb_{i_1 \dots i_m j_1 \dots j_m}  \, f_{j_1 \dots j_m} (y) \,  \frac{|y-x|^{l}}{|y-x|^{n-1}} \, \D y.
\end{aligned}
\end{equation*}
Note that for $x, z \in \Rb^n$, we can write $\langle x, z\rangle^r = j_{x^{\odot r}} z^{\odot r}$. % \L x^{{\otimes}r}, z^{\otimes r} \R$, where the right hand side is interpreted as a dot product of the vectors $x^{{\otimes}r}$ and $z^{{\otimes}r}$. 
Then we can write 
\begin{align*}
	(N_m^k f)_{i_1 \dots i_m} (x) &= 2 \sum\limits_{l=0}^k \binom{k}{l}  \int\limits_{\Rb^n} \lb j_{x^{\odot 2k-l}}(y-x)^{\odot 2m+2k-l}\rb_{i_1 \cdots i_m j_1 \cdots j_m} \frac{f_{j_1 \dots j_m}(y)}{|y-x|^{2m+2k-2l +n-1}} \, \D y.
\end{align*}
%Recall that we use the notation $\lb x^{\odot 2k-l}\rb_{p_1 \dots p_{2k-l}}$ to denote the $p_1 \cdots p_{2k-l}$ component of the tensor $x^{\odot 2k-l}$, and similar interpretation for $(y-x)^{\odot 2m+2k-l}_{p_1 \cdots p_{2k-l} i_1 \dots i_m j_1 \dots j_m}$. 
This gives
\begin{equation}\label{nmrt}
	(N_m^k f)_{i_1 \dots i_m} (x) = 2 \sum\limits_{l=0}^k \binom{k}{l} (-1)^l x^{\odot 2k-l}_{p_1 \cdots p_{2k-l}} \left[ f_{j_1 \dots j_m} * \frac{\lb x^{\odot 2m+2k-l}\rb_{p_1 \cdots p_{2k-l} i_1 \cdots i_m j_1 \dots j_m}}{|x|^{2m+2k-2l+n-1}} \right].
\end{equation}
Equation \eqref{nmrt} makes sense for $f\in\Ec'(\Rb^n; S^m)$ as well. For $f\in  \Ec'(\Rb^n; S^m)$, $N_m^k: \Ec'(\Rb^n; S^m) \to \Dc'(\Rb^n; S^m)$ can be viewed as a multiplication of a smooth function with the convolution of a compactly-supported distribution and a tempered distribution.

We will use in our calculations the divergence of the normal operator of momentum ray transforms given by 
\begin{equation}
	(\delta N_m^k f)_{i_1 \dots i_{m-1}} (x) = k \, \int\limits_{\mathbb{S}^{n-1}} \langle x, \xi\rangle^{k-1} \, \xi_{i_1} \dots \xi_{i_{m-1}} \, J^k f(x-\langle x, \xi\rangle \xi, \xi) \, \Dxi.
\end{equation}
This can be obtained by directly applying the divergence operator to \eqref{Normal}. Iterating, we get the formula

\begin{align}
	(\delta^r N_m^k f)_{i_1 \dots i_{m-r}} (x) &= \frac{k!}{(k-r)!} \, \int\limits_{\mathbb{S}^{n-1}} \langle x, \xi\rangle^{k-r} \, \xi_{i_1} \dots \xi_{i_{m-r}} \, J^k f(x-\langle x,\xi\rangle \xi, \xi) \, \Dxi \notag
	\\ &= \frac{k!}{(k-r)!} \, \sum\limits_{l=0}^k \binom{k}{l} \int\limits_{\mathbb{S}^{n-1}} \langle x, \xi\rangle^{2k-r-l} \, \xi_{i_1} \dots \xi_{i_{m-r}} \, J^l f(x, \xi) \, \Dxi. \notag
\end{align}
In particular, 
\begin{equation}\label{div_nmrt}
	(\delta^k N_m^k f)_{i_1 \dots i_{m-k}} (x) = k! \, \sum\limits_{l=0}^k \binom{k}{l} \int\limits_{\mathbb{S}^{n-1}} \langle x, \xi\rangle^{k-l} \, \xi_{i_1} \dots \xi_{i_{m-k}} \, J^l f(x, \xi) \, \Dxi, \end{equation}
and
\begin{equation}
	\delta^{k+1} N_m^k f = 0.
\end{equation}

To conclude this section, we remark that our approach for proving a unique continuation principle for the transverse ray transform on symmetric tensor fields is based on the analysis for the ray transform of scalar functions/distributions, and since we already know to handle this case, we do not define transverse ray transform of symmetric tensor field distributions separately.

We are now ready to state the main results of the article. 

We say a function \emph{$\psi$ vanishes to infinite order at a point $x_0 \in \Rb^n$} if $\psi$ is smooth in a neighborhood of $x_0$ and $\psi$ along with its partial derivatives of all orders vanishes at $x_0$, that is, $\p^\alpha_{x} \psi (x_0) = 0$ for all multi-indices $\alpha$.

\begin{theorem}[UCP for ray transform I]\label{ucp_nrt}
	Let $U \subseteq \Rb^n$ be any non empty open set and $ n\ge 2$. Let $f \in \Ec'(\Rb^n; S^m)$ be such that $Rf|_{U} = 0$ and $N_mf$ vanishes to infinite order at some $x_0 \in U$. Then $f$ is a potential field, that is,  there exists a $v \in \Ec'(\Rb^n; S^{m-1})$ such that $f = \D v$. 
\end{theorem}
\begin{theorem}[UCP for ray transform II]\label{ucpnrt1}
	Let $U \subseteq \Rb^n$ be any non empty open set and $ n\ge 2$. Let $f \in \Ec'(\Rb^n;S^m)$ be such that $Rf = 0$ and $N_mf=0 $ in $ U$. Then $f$ is a potential field. %
\end{theorem}

\begin{corollary}[UCP for ray transform III]\label{ucp_rt}
	Let $U \subseteq \Rb^n$ be open. Let $f \in \Ec'(\Rb^n; S^m)$ be such that $Rf|_U = 0$ and the ray transform of $f$ vanishes on all lines passing through $U$, that is, $J_mf(x, \xi) = 0$ for $x \in U,\, \xi \in \Sn$. Then $f$ is a  potential field. %, that ,  there exists $w \in \Ec'(S^{m-1}; \Rb^n)$ such that $f = \D w$. 
\end{corollary}

\begin{theorem}[UCP for momentum ray transform I]\label{mrt_ucp}
	Let $U \subseteq \Rb^n$. Let $f \in \Ec'(\Rb^n; S^m)$ be such that for some $0 \leq k \leq m$, $R^k f|_U = 0$. If $N_m^p f|_U = 0$ for all $0 \leq p \leq k$, then 
	 $f$ is a generalized potential field, that is, there exists a $v\in \Ec'(\Rb^n; S^{m-k-1})$ such that  $f= \D ^{k+1}v$.
	
	%there exists a field $v \in \Ec'(\Rb^n; S^{m-k-1})$ such that $f = \D ^{k+1}v$ and the support of $ v$ is  contained in the convex hull of support of $f$.
\end{theorem}

\begin{theorem}[UCP for momentum ray transform II]\label{mrt_ucp_1}
	Let $U \subseteq \Rb^n$ and $f \in \Ec'(\Rb^n; S^m)$. Suppose for some $0 \leq k \leq m$, $R^k f|_U = 0$ and $J^k_{m} f(x,\xi) = 0$ for all $(x,\xi)\in U\times \mathbb{S}^{n-1}$, then $f$ is a generalized potential tensor field.   % and support of $ v$ is  contained in the convex hull of support of $f$.
\end{theorem}
In all the results above the support of $v$ is contained in the convex hull of the support of $f$.
\begin{theorem}[UCP for transverse ray transform]\label{ucp_trt}
	Let $ n\ge 3$ and $f\in \Ec'(\Rb^n; S^m)$. % be a compactly supported tensor field distribution. 
	Assume that $\Tc f=0$ along all the lines intersecting a non-empty open set $U$ and $f=0$ in $U$. Then $f\equiv 0$. 
\end{theorem}

\section{ucp for the ray transform}\label{UCP:Ray}%In this section we state our main results.

%%%mention that the result can be made stronger

%We would like to mention here that we do not need the full strength of the assumption $N_m f = 0$ in $U$. We only need that $N_m f$ \emph{vanishes to infinite order at a point in U} as follows.

In this section, we prove unique continuation properties for the ray transforms. We prove Theorem \ref{ucp_nrt}.  We first show that $N_mf$ is smooth in $U$ if $Rf|_U = 0$.

\begin{lemma}\label{smoothnesslemma}
    Let $f \in \Ec'(\Rb^n; S^m)$ be such that $Rf|_{U} = 0$ for some open set $U \subseteq \Rb^n$. Then,  $N_mf|_U$ is smooth.
\end{lemma} 
\begin{proof}
    We use the following formula proved in \cite[Theorem 2.5]{Denisjuk_Paper} for $f\in \Sc(\Rb^n; S^m)$: 
    \begin{equation}\label{smoothness}
        \Delta^m (^sf) = 2^m \delta_e^m Rf,
    \end{equation}
    where $\delta_e$ is the even indices divergence operator \cite[(2.4)]{Denisjuk_Paper} and $^s\!f$ is the solenoidal component of $f$. In \eqref{smoothness}, the formula is interpreted componentwise. The same proof works for  $f \in \Ec'(\Rb^n;S^m)$ as well. By hypothesis, the right hand side  of \eqref{smoothness} is $0$ in $U$, and hence $\Delta^m \left( ^s\!f \right)=0$. By Weyl's Lemma \cite{Weyl_lemma}, $^s\!f$ is smooth in $U$, and hence so is $N_m \,^s\!f$.  Since $N_m f = N_m ^s f$;   see \cite{Sharafutdinov_Book}, we are done.
\end{proof}

% \begin{theorem}\label{decomposition}\cite[Theorem 2.14.1]{Sharafutdinov_Book}
%     Let $n \geq 2$. For $f \in \Ec'(S^m; \Rb^n)$, there exist uniquely determined fields $^sf \in \Sc'(S^m; \Rb^n)$ and $v \in \Sc'(S^{m-1}; \Rb^n)$ smooth outside support $f$, tending to zero at infinity and satisfying
%     \begin{align*}
%         f = ^sf + dv, \quad \d ^sf = 0.
%     \end{align*}
% \end{theorem}

%Now we know that $N_m f$ is actually a smooth function in $U$, we can strengthen Theorem~\ref{ucpnrt1}. 

%\ssc{We have to mention}
%We will use the decomposition in Theorem~\ref{decomposition} and show that the solenoidal component of $f$ vanishes identically. \footnote{this may not be the case} By Lemma~\ref{smoothnesslemma}, the normal operator is smooth in $U$, and thus the point-wise derivatives make sense. 

%At the level of ray transform, we have the following uniqueness result.

Now we will only prove Theorem~\ref{ucp_nrt}. Theorem~\ref{ucpnrt1} is a trivial consequence. The idea of proof is to reduce the problem to that of the functions case as in \cite{Keijo_partial_vector_field}.  The following proposition serves as the key ingredient.
\begin{proposition}\label{key_lemma_RT}
    For $f \in \Ec'(\Rb^n; S^m)$, the following equality holds: 
    \begin{equation}\label{key_equation_RT}
            m!\, N_0(Rf)_{i_1 j_1 \dots i_m j_m} = \sum_{l=0}^{\lfloor\frac{m}{2}\rfloor} c_{l,m} (R(i^l j^l N_mf))_{i_1 j_1 \dots i_m j_m},
    \end{equation}
  where $ \lfloor x \rfloor$ denotes the greatest integer function $\leq x$ and the constants $c_{l,m}$ are given by
    \begin{equation*}
        c_{l,m} = \left \{\prod_{p=0}^{m-l-1} (n-1+2p)\right \} \frac{(-1)^l \, m!}{2^l\, l! \,  (m-2l)!}.
    \end{equation*}
\end{proposition}
\begin{remark}
    An inversion formula recovering the Saint-Venant operator of $f$ from its ray transform is proved in \cite[Theorem 2.12.3]{Sharafutdinov_Book}. From \eqref{key_equation_RT}, one can derive this inversion formula. However, our approach here is  different and in our opinion simpler than that of \cite[Theorem 2.12.3]{Sharafutdinov_Book}. % the proof given in this article is simpler and different from the one presented in \cite{Sharafutdinov_Book}. }
\end{remark}
To prove this proposition we first prove a technical result, see Lemma \ref{IBPlemma},  which will be used in the next two sections. We first recall the definition of a positive homogeneous function.

%\begin{definition}
A function $g$ is called \emph{positive homogeneous of degree $\lambda$} if 
\begin{equation*}
    g(rx) = r^\lambda g(x)
\end{equation*}
for all $x \in \mathbb{R}^n$, and $r > 0$.
%\end{definition}
%By \eqref{homo_mrt1}, $I_mf$ is positive homogenous of degree $m-1$ in the $\xi$-variable. We prove a scaling property of homogenous functions that we will need.

The following lemma was recently used in \cite{krishnan2021ray} as well. 

\begin{lemma}\label{homogenous_property}
If $g$ is smooth and positive homogeneous of degree $\lambda$ such that $n+ \lambda > 0$, then
\begin{equation*}
    \int\limits_{\Sn} g(\xi)\, \D \mathrm{S}_\xi  = (n+\lambda) \int\limits_{|\xi|\leq 1} g(\xi)\, \D \xi.
\end{equation*}
\end{lemma}
\begin{proof}
Using polar coordinates, we have
\begin{align*}
    \int\limits_{|\xi|\leq 1} g(\xi)\, \D \xi  &= \int\limits_0^1 \int\limits_{\Sn} g(r\xi) r^{n-1}\, \D \mathrm{S}_\xi\,  \D r
    \\ &= \int\limits_0^1 \int\limits_{\Sn} r^{n+ \lambda-1} g(\xi)\D \mathrm{S}_\xi \, \D r
    \\ &= \frac{1}{(n+\lambda)}\, \int\limits_{\Sn} g(\xi) \D \mathrm{S}_\xi.
\end{align*}
\end{proof}

\begin{lemma}\label{IBPlemma}
    Let $n \geq 2$ and $g$ be a smooth function on $\mathbb{R}^n$ such that $g$ is positive homogeneous of degree $s-1$ for some $s \in \mathbb{N}$. Then the following equality holds: 
    \begin{equation}\label{IBP}
        \int\limits_{\Sn} \frac{\p^s g}{\p \xi_{i_1} \dots \p \xi_{i_s}} \Dxi = \sum\limits_{l=0}^{ \lfloor \frac{s}{2}\rfloor} c_{l,s}  \int\limits_{\mathbb{S}^{n-1}} i^l j^l\lb \xi^{\odot s}\rb _{i_1 \cdots i_s} \, g(\xi) \, \Dxi ,
    \end{equation}
    with the constants $c_{l,s}$   given by
    \begin{align}\label{constants}
        c_{l,s} = \prod\limits_{w=0}^{(s-l-1)}(n-1+2w) \frac{(-1)^l \, s! }{2^l \, l! \, (s-2l)!}.
    \end{align}
\end{lemma}
\begin{proof}
    The proof proceeds by induction on $s$. %We divide the induction argument into several steps.\\
We first prove \eqref{IBP} for $s=1$.
    
   For a smooth function $g$ positive homogeneous of degree $0$, we have from Lemma \ref{homogenous_property},
    \[
    \int\limits_{\Sn} \frac{\partial g}{\partial \xi_{i}} \, \Dxi = (n-1)\int\limits_{|\xi|\leq 1}  \frac{\partial g}{\partial \xi_{i}} \, \D \xi,
    \]
    since $\frac{\PD g}{\PD \xi_i}$ is positive homogeneous of degree $-1$. Now applying the divergence theorem, we get, 
    \[
    \int\limits_{|\xi|\leq 1} \frac{\partial g}{\partial \xi_{i}} \, \D \xi = \int\limits_{\Sb^{n-1}} \xi_i g(\xi) \Dxi.
 \]
 Hence 
 \Beq\label{s=1case}
 \int\limits_{\Sn} \frac{\partial g}{\partial \xi_{i}} \, \Dxi = (n-1)\int\limits_{\Sb^{n-1}} \xi_i g(\xi) \Dxi.
  \Eeq  
 
For $s=1$, the only choice for $l$ is $0$ and  $c_{0,1}=n-1$ in \eqref{constants}. Thus for $ s=1$ \eqref{IBP} agrees with \eqref{s=1case}. \\

Now assume that \eqref{IBP}  is true for some $s = r-1$. We want to show that the equality \eqref{IBP} holds for  $s=r$.

   % Assume that \eqref{IBP}  is true for some $s = r-1$. We want to show that the equality \eqref{IBP} holds for  $s=r$.
Let  $g$ be a positive homogeneous of degree $r-1$. Then, $\frac{\partial g}{\partial \xi_{i_r}}$ is positive homogeneous of degree $r-2$ and by induction hypothesis we obtain
        \begin{align*}
            \int\limits_{\mathbb{S}^{n-1}} \frac{\partial^r g}{\partial \xi_{i_1} \dots \partial \xi_{i_r}} \, \Dxi  = \sum\limits_{l=0}^{\lfloor \frac{r-1}{2} \rfloor}c_{l,r-1} \left(\,\int\limits_{\mathbb{S}^{n-1}} i^l j^l \lb \xi^{\odot r-1}\rb_{i_1 \cdots i_{r-1}}  \frac{\partial g}{\partial \xi_{i_r}} \, \Dxi \right).
        \end{align*}
    %where the operators $i^l j^l$ are written inside the integral because they act only on the indices $\xi_{i_1} \dots \xi_{i_{r-1}}$.
    Note that $ i^l j^l (\xi^{\odot r-1}) \, \frac{\partial g}{\partial \xi_{i_r}} $ is a positive homogeneous function of degree $ 2r-2l-3$. % on $ \mathbb{S}^{n-1}$. 
    Applying Lemma \ref{homogenous_property} and using Gauss divergence theorem, we get
    \begin{align}
       \label{r-derivative integral} \int\limits_{\mathbb{S}^{n-1}} \frac{\partial^r g}{\partial \xi_{i_1} \dots \partial \xi_{i_r}} \, \Dxi &= \sum\limits_{l=0}^{\lfloor \frac{r-1}{2} \rfloor}c_{l,r-1} (n+2r-2l-3) \Bigg[ -\int\limits_{|\xi|\leq 1}\left( \frac{\partial}{\partial \xi_{i_r}}i^l j^l \lb \xi^{\odot r-1}\rb_{i_1 \cdots i_{r-1}} \rb  \, g(\xi) \, \D \xi
       \\ \notag & \qquad +  \int\limits_{\mathbb{S}^{n-1}} i^l j^l \lb \xi^{\odot r-1}\rb_{i_1 \cdots i_{r-1}} \, \xi_{i_r}\,  g(\xi) \, \Dxi \Bigg].
       \end{align}
      We now use Lemma \ref{homogenous_property}  in the first integral in \eqref{r-derivative integral} and obtain  
       \begin{equation}\label{eq_3.7}
             \begin{aligned}
      \int\limits_{\mathbb{S}^{n-1}} \frac{\partial^r g}{\partial \xi_{i_1} \dots \partial \xi_{i_r}} \, \Dxi  &= \sum\limits_{l=0}^{\lfloor \frac{r-1}{2} \rfloor}c_{l,r-1} (n+2r-2l-3) \int\limits_{\mathbb{S}^{n-1}} i^l j^l \lb \xi^{\odot r-1}\rb_{i_1 \cdots i_{r-1}}  \xi_{i_r} g(\xi) \, \Dxi
        \\  & \qquad - \sum\limits_{l=0}^{\lfloor \frac{r-1}{2} \rfloor}c_{l,r-1} 
        \int\limits_{\mathbb{S}^{n-1}} \lb \frac{\partial}{\partial \xi_{i_r}}i^l j^l \lb \xi^{\odot r-1}\rb_{i_1 \cdots i_{r-1}} \rb  g(\xi) \, \Dxi,
    \end{aligned}
    \end{equation}
 We separate the $l=0$ term from the first sum, and $ l = \lfloor \frac{r-1}{2} \rfloor$ term from the second sum in \eqref{eq_3.7} to get
    \begin{equation}\label{IBPlemma_eq1}
        \begin{split}
            \int\limits_{\mathbb{S}^{n-1}} \frac{\partial^r g}{\partial \xi_{i_1} \dots \partial \xi_{i_r}} \, \Dxi &= c_{0,r-1}(n+2r-3) \int\limits_{\mathbb{S}^{n-1}} \xi_{i_1} \dots \xi_{i_r} \, g(\xi) \, \Dxi
            \\ & \quad + \sum\limits_{l=1}^{\lfloor \frac{r-1}{2} \rfloor}c_{l,r-1}(n+2r-2l-3)  \int\limits_{\mathbb{S}^{n-1}}  i^l j^l \lb \xi^{\odot r-1}\rb_{i_1 \cdots i_{r-1}}  \xi_{i_r} \, g(\xi) \, \Dxi
            \\& \quad-   \sum\limits_{l=0}^{\lfloor \frac{r-1}{2} \rfloor -1}c_{l,r-1}
            \int\limits_{\mathbb{S}^{n-1}} \lb \frac{\partial}{\partial \xi_{i_r}} i^l j^l \lb \xi^{\odot r-1}\rb_{i_1 \cdots i_{r-1}} \rb g(\xi) \, dS(\xi) 
            \\ &\quad -  c_{\lfloor \frac{r-1}{2} \rfloor,r-1} \int\limits_{\mathbb{S}^{n-1}} \lb \frac{\partial} {\partial \xi_{i_r}}\, i^{\lfloor \frac{r-1}{2} \rfloor} j^{\lfloor \frac{r-1}{2} \rfloor} \lb\xi^{\odot r-1}\rb_{i_1 \cdots i_{r-1}}\rb  g(\xi) \, \Dxi.
           \end{split}
    \end{equation}
   We analyze the integrals in  \eqref{IBPlemma_eq1}  separately. 
   
     Since $c_{0,r-1}(n+2r-3) = c_{0,r}$, we have 
    \begin{equation}\label{IBPlemma_eq2}
        c_{0,r-1} (n+2r-3) \int\limits_{\mathbb{S}^{n-1}} \xi_{i_1} \dots \xi_{i_r} \, g(\xi) \, \Dxi = c_{0,r} \int\limits_{\mathbb{S}^{n-1}} \xi_{i_1} \dots \xi_{i_r} \, g(\xi) \, \Dxi.
    \end{equation}

Next we analyze the last term in \eqref{IBPlemma_eq1}. First we consider the case of odd $r$. Letting $r=2k+1$, we have $\lfloor \frac{r-1}{2} \rfloor = k$. Then we have 
\[
i^{k}j^{k} \lb \xi^{\odot 2k}\rb_{i_1 \cdots i_{2k}}=\sigma(i_1\cdots i_{2k})\lb\delta_{i_1 i_2} \cdots \delta_{i_{2k-1} i_{2k}}\rb.
\]
Hence for $r$ odd, the last term in \eqref{IBPlemma_eq1} is $0$.

Next we consider the case when $r$ is even. Letting $r=2k$, we have  $\lfloor \frac{r-1}{2} \rfloor = k-1=\lfloor \frac{r}{2} \rfloor - 1$. This implies
\begin{align*}
	c_{\lfloor \frac{r-1}{2} \rfloor, r-1} &=  \prod\limits_{w=0}^{r-k-1}(n-1+2w) \frac{(-1)^{k-1} \, (r-1)! }{2^{k-1} \, (k-1)! \, 1!}\\
%	\intertext{Multiplying and dividing by $r$ and writing $2^{k-1}r$ as $2^k (r/2)$, we get}
	&=  -\prod\limits_{w=0}^{r-k-1}(n-1+2w) \frac{(-1)^{k} \, r! }{2^{k} \, k! }\mbox{ using} \quad r=2k,\\
	&= - c_{k,r}=-c_{\lfloor \frac{r}{2} \rfloor, r}.
\end{align*}
We have 
\begin{align*}
	\frac{\PD}{\PD\xi_{i_r}}\lb i^{k-1} j^{k-1}\lb \xi^{\odot 2k-1}\rb_{i_1 \cdots i_{2k-1}}\rb&=\frac{\PD}{\PD\xi_{i_r}} \sigma(i_1 \cdots i_{2k-1})\lb\delta_{i_1 i_2} \cdots \delta_{i_{2k-3} i_{2k-2}} \xi_{i_{2k-1}}\rb\\
	&=\sigma(i_1 \cdots i_{2k-1})\lb\delta_{i_1 i_2} \cdots \delta_{i_{2k-3} i_{2k-2}} \delta_{i_{2k-1} i_{2k}}\rb. 
	\end{align*}
We observe, recalling that $r=2k$, 
\[
\sigma(i_1 \cdots i_{r-1}) \{ \delta_{i_1 i_2} \cdots \delta_{i_{r-1}i_r}\} = \sigma(i_1 \cdots i_r) \{\delta_{i_1 i_2} \cdots \delta_{i_{r-1} i_r}\}.
\]
Therefore, 
\[
	\frac{\PD}{\PD\xi_{i_r}}\lb i^{k-1} j^{k-1}\lb \xi^{\odot 2k-1}\rb_{i_1 \cdots i_{2k-1}}\rb= i^{k}j^{k} \lb \xi^{\odot 2k}\rb_{i_1\cdots i_{2k}}.
\]
Summarizing, we have, 
    \begin{equation}\label{IBPlemma_eq3}
        \begin{split}
            -  c_{\lfloor \frac{r-1}{2} \rfloor,r-1} \int\limits_{\mathbb{S}^{n-1}} \frac{\partial} {\partial \xi_{i_r}}\lb i^{\lfloor \frac{r-1}{2} \rfloor} j^{\lfloor \frac{r-1}{2} \rfloor}\lb \xi^{\odot r}\rb_{i_1\cdots i_r} \rb g(\xi) \Dxi
            \\ =    \begin{cases}  0 &\mbox{if $r$ is odd} \\ c_{\lfloor \frac{r}{2} \rfloor, r} \int\limits_{\mathbb{S}^{n-1}} i^{\lfloor \frac{r}{2} \rfloor} j^{\lfloor \frac{r}{2} \rfloor} \lb \xi^{\odot r}\rb_{i_1 \cdots i_r} g(\xi) \, \Dxi & \text{if $r$ is even.}  \end{cases} 
        \end{split}
    \end{equation}
    %In the remaining terms, we shift the indices on the second summation on the left below to get
    We now consider the remaining terms in \eqref{eq_3.7}:
    \begin{align*}
    & \sum\limits_{l=1}^{\lfloor \frac{r-1}{2} \rfloor}c_{l,r-1}(n+2r-2l-3)  \int\limits_{\mathbb{S}^{n-1}}  i^l j^l \lb \xi^{\odot r-1}\rb_{i_1 \cdots i_{r-1}}  \xi_{i_r} \, g(\xi) \, \Dxi
    \\ &-  \sum\limits_{l=0}^{\lfloor \frac{r-1}{2} \rfloor -1}c_{l,r-1}
    \int\limits_{\mathbb{S}^{n-1}} \lb \frac{\partial}{\partial \xi_{i_r}} i^l j^l \lb \xi^{\odot r-1}\rb_{i_1 \cdots i_{r-1}} \rb g(\xi) \, \Dxi = I \quad\mbox{(say)}
    \end{align*}
    Re-indexing the second integral above we obtain
    \begin{equation}\label{IBPlemma_eq4}
        \begin{aligned}
    I
    &=\sum\limits_{l=1}^{\lfloor \frac{r-1}{2} \rfloor}c_{l,r-1}(n+2r-2l-3)  \int\limits_{\mathbb{S}^{n-1}}  i^l j^l \lb \xi^{\odot r-1}\rb_{i_1 \cdots i_{r-1}}  \xi_{i_r} \, g(\xi) \, \Dxi
    \\ &\quad-  \sum\limits_{l=1}^{\lfloor \frac{r-1}{2} \rfloor }c_{l-1,r-1}
    \int\limits_{\mathbb{S}^{n-1}} \lb \frac{\partial}{\partial \xi_{i_r}} i^{l-1}j^{l-1} \lb \xi^{\odot r-1}\rb_{i_1 \cdots i_{r-1}} \rb g(\xi) \, \Dxi.
    \end{aligned}
    \end{equation}
    
   % \end{align}    
   % \begin{comment}
  %  \begin{align}
        %\sum\limits_{l=1}^{\lfloor \frac{r-1}{2} \rfloor}c_{l,r-1}(n+2r-2l-3)  \int\limits_{\mathbb{S}^{n-1}} i^l j^l (\xi_{i_1} \dots \xi_{i_{r-1}}) \xi_{i_r} \, g(\xi) \, dS(\xi) 
    %    \\- \sum\limits_{l=1}^{\lfloor \frac{r-1}{2} \rfloor} c_{l-1,r-1}
    %    \int\limits_{\mathbb{S}^{n-1}} \frac{\partial}{\partial \xi_{i_r}}\, i^{l-1} j^{l-1}( \xi_{i_1} \dots \xi_{i_{r-1}}) \, g(\xi) \, dS(\xi) \notag
   % \end{align}
   % \end{comment}
% Note that $c_{l,r-1} (n+2r-2l-3) = \frac{r-2l}{r} \, c_{l,r}$ and $-c_{l-1,r-1} = \frac{2l}{r(r-2l+1) c_{l,r}$, which further transform \eqref{IBPlemma_eq4}  to
Note that $c_{l,r-1} (n+2r-2l-3) = \frac{r-2l}{r} \, c_{l,r}$ and $-c_{l-1,r-1} = \frac{2l}{r(r-2l+1)} c_{l,r}$. Combining this with \eqref{IBPlemma_eq4} we get
\begin{align}
         I&= \sum\limits_{l=1}^{\lfloor \frac{r-1}{2} \rfloor} \frac{c_{l,r}}{r} \Bigg\{ (r-2l) \int\limits_{\mathbb{S}^{n-1}} i^l j^l \lb \xi^{\odot r-1}\rb_{i_1 \cdots i_{r-1}} \xi_{i_r} \, g(\xi) \, \Dxi
      \notag  \\  & \quad+ \frac{2l}{r-2l+1} \int\limits_{\mathbb{S}^{n-1}} \lb \frac{\partial}{\partial \xi_{i_r}}\, i^{l-1} j^{l-1}\lb \xi^{\odot r-1}\rb_{i_1 \cdots i_{r-1}} \rb g(\xi) \, \Dxi \Bigg \}. \notag
    \end{align}
    Finally, we observe the following:
    for $0 \leq l \leq \lfloor \frac{r-1}{2} \rfloor$,
    
    \begin{align}
      \label{IBPlemma_eq5}  \int\limits_{\mathbb{S}^{n-1}} i^l j^l \xi^{\odot r}_{i_1 \cdots i_{r}}\, g(\xi) \, d\xi &= \frac{(r-2l)}{r} \int\limits_{\mathbb{S}^{n-1}}  i^l j^l  \lb \xi^{\odot r-1}\rb_{i_1 \cdots i_{r-1}}\xi_{i_r} \, g(\xi) \, \Dxi
        \\  \notag \quad&+ \frac{2l}{r(r-2l+1)} \int\limits_{\mathbb{S}^{n-1}} \frac{\partial}{\partial \xi_{i_r}}\lb  i^{l-1} j^{l-1} \lb \xi^{\odot r-1}\rb_{i_1 \cdots i_{r-1}} \rb \, g(\xi) \, \Dxi.
    \end{align}
    %where the operators $i$ and $j$ on the LHS act on all the indices $i_1 , \dots, i_r$ but only on $i_1, \dots , i_{r-1}$ on the RHS.
    
    This can be seen by expanding the left hand side expression,
    \begin{align*}
    	 i^l j^l \lb \xi^{\odot r}\rb_{i_1 \cdots i_{r}}= \sigma(i_1 \cdots i_r)\lb \delta_{i_1 i_2} \cdots \delta_{i_{2l-1} i_{2l}} \xi_{i_{2l+1}}\cdots \xi_{i_r}\rb.
    	\end{align*}
   The first term comes from those permutations that takes $i_r$ to one of $i_{2l+1} \cdots i_r$ and the second term comes from the complement.

    %This can be seen by separating out the permutations that take $i_r$ 
    
    Finally, substituting \eqref{IBPlemma_eq2}, \eqref{IBPlemma_eq3} and \eqref{IBPlemma_eq5} into \eqref{IBPlemma_eq1} we get \eqref{IBP} for $s=r$. This completes the  proof. %induction step. 
    %Combining \emph{Step 1} and \emph{Step 2} we complete the proof of Lemma \ref{IBPlemma}.
\end{proof}

%We begin with the proof of Proposition~\ref{key_lemma_RT}.

\begin{proof}[Proof of Proposition~\ref{key_lemma_RT}]
    We prove \eqref{key_equation_RT} for $f \in C_c^{\infty}(\Rb^n; S^m)$ first.
    
    By an iteration of \cite[Equation~2.10.2]{Sharafutdinov_Book}
    \begin{equation}\label{john_ray_relation}
        (-2)^m m! \, I_0((Rf)_{i_1 j_1 \cdots i_m j_m}) = J_{i_1 j_1} \cdots J_{i_m j_m} (J_m f),
    \end{equation} 
%    \dbc{some justification needed}
    where the ray transform acts on the scalar function $(Rf)_{i_1 j_1 \cdots i_m j_m}$ on the left hand side and for each $ 1\le i,j\le n$  the John operator \cite{John_diff_paper} $
    J_{ij}$ is  defined,  for functions $\vp \in C^{\infty}(\Rb^n \times \Rb^n\backslash\{0\})$, by
    \begin{equation}\label{John_operator}
         J_{ij}\vp(x,\xi)  = \left(\frac{\p^2 \vp}{\p x_i \p \xi_j} - \frac{\p^2 \vp}{\p x_j \p \xi_i}\right).
    \end{equation}
     Integrating both sides of \eqref{john_ray_relation} over $\Sn$,
    \begin{align*}
        (-2)^m m! N_0((Rf)_{i_1 j_1 \dots i_m j_m}) &= 2^{m}\alpha(i_1 j_1) \dots \alpha(i_m j_m) \frac{\p^m}{\p x_{i_1} \dots \p x_{i_m}} \int\limits_{\Sn} \frac{\p^m J_m f(x,\xi)}{\p \xi_{j_1} \dots \p \xi_{j_m}} \, \Dxi.
        \end{align*}
        Using Lemma \ref{IBPlemma},
        \begin{align*}
            &(-2)^m m! N_0((Rf)_{i_1 j_1 \dots i_m j_m})\\ &= 2^{m}\alpha(i_1 j_1) \dots \alpha(i_m j_m) \frac{\p^m}{\p x_{i_1} \dots \p x_{i_m}} \sum\limits_{l=0}^{\lfloor \frac{m}{2} \rfloor} c_{l,m}  \big(\int\limits_{\mathbb{S}^{n-1}} \bigg(i^l j^l\lb \xi^{\odot m}\rb_{j_1 \cdots j_m} J_m f(x,\xi) \, \Dxi \big)\bigg)\\
        &=2^{m}\sum\limits_{l=0}^{\lfloor \frac{m}{2} \rfloor} c_{l,m} R\lb i^{l} j^{l} N_{m}f\rb_{i_1 j_1 \cdots i_m j_m}.    \end{align*}
   % where we used Lemma \ref{IBPlemma}. 
    Now using anti-symmetry of $R$ on the right hand side, we get \eqref{key_equation_RT}.  The proof for $f\in C_c^{\infty}(\Rb^n; S^m)$ is complete.  The proof for $f \in \Ec'(\Rb^n; S^m)$ follows by integral representation for the normal operator $N_m$; see \eqref{NormalOperator_Ray},  the density of $C_c^{\infty}(\Rb^n; S^m)$ in $\Ec'(\Rb^n; S^m)$ combined with the fact that convolution: \[
    (u,v)\to u*v \mbox{ from }
    \Ec'(\Rb^n)\times \Dc'(\Rb^n) \to \Dc'(\Rb^n), 
    \]
    is a separately continuous bilinear map. See \cite[Theorem 27.6]{TrevesBook}. %Indeed, the r %from $\Ec'(\Rb^nas we show in Lemma \ref{lm:density} below.
    %The above expression is nothing but \eqref{key_equation_RT}.
    
     %The proof for $f \in \Ec'(S^m; \Rb^n)$ follows by  density of $\Dc(S^m; \Rb^n)$ in $\Ec'(S^m; \Rb^n)$.%%%%%%%% Let $f \in \Ec'(\Rb^n)$ and $f_n \in \Dc(\Rb^n)$ such that $f_n \xrightarrow[n \to \infty]{\Ec'(\Rb^n)} f$. Let $\psi \in \Sc'(\Rb^n)$ and $\varphi \in \Dc(\Rb^n)$. Then, $\psi * \varphi \in \Ec(\Rb^n)$, which implies $f_n (\psi * \varphi) \xrightarrow[n \to \infty]{} f(\psi * \varphi)$. Thus, $f_n * \psi \xrightarrow[n \to \infty]{\Dc'(\Rb^n)}f * \psi$ for any $\psi \in \Sc'(\Rb^n)$. By commutativity of convolution, $\psi * f_n \xrightarrow[n \to \infty]{\Dc'(\Rb^n)} \psi * f$. Since $R$ is a differential operator, which is continuous in the space of distributions, and the normal operators $N_m$ for $m\geq 0$ are given by convolution with tempered distributions, \eqref{key_equation_RT} follows for $f \in \Ec'(S^m; \Rb^n)$.  
\end{proof}

With the preliminary results in place, we prove the unique continuation results for the ray transform of symmetric tensor fields below.

\begin{proof}[Proof of Theorem~\ref{ucp_nrt}]
    By Lemma~\ref{smoothnesslemma}, the hypothesis $Rf|_U = 0$ implies that the normal operator $N_m f$ is smooth in $U$. Also, $N_m f$ vanishes to infinite order at $x_0 \in U$. Thus, by \eqref{key_equation_RT}, $N_0(Rf)_{i_1 j_1 \dots i_m j_m}$ is smooth in $U$ and vanishes to infinite order at $x_0 \in U$. Using \cite[Theorem 1.1]{Keijo_partial_function}, we conclude that $(Rf)_{i_1 j_1 \dots j_m j_m} \equiv 0$. Since this is true for all indices $i_1, j_1, \dots i_m, j_m$, $Rf$ vanishes identically on $\Rb^n$, which is equivalent to $Wf$ vanishing on $\Rb^n$. Thus by \cite[Theorem 2.5.1]{Sharafutdinov_Book}, there exists a field $v \in \Ec'(\Rb^n; S^{m-1})$ whose support is contained in the convex hull of  support of $f$ and $f = \D v$. %\footnote{$\Ec'(S^{m-1})$ or $\Sc'(S^{m-1})$?}
\end{proof}

The proof of Theorem \ref{ucpnrt1} follows as a direct consequence of Theorem \ref{ucp_nrt}.

\begin{proof}[Proof of Corollary~\ref{ucp_rt}]
    Since $J_m f (x,\xi) = 0$ for $x \in U$ and $\xi \in \Sn$, using homogeneity properties, we know $I_m f\lb x-\langle x,\xi\rangle \xi, \xi\rb$ and hence we have that $N_m f (x) = 0$ for all $x \in U$. This in particular implies that $N_m f$ is smooth in $U$ and, trivially, vanishes to infinite order at any point $x_0 \in U$. The conclusion follows from Theorem \ref{ucp_nrt}. 
\end{proof}

%Corollary~\ref{ucp_rt} follows by Theorem~\ref{ucp_nrt} as stated above. However 
We also present a much simpler proof working directly with the ray transform.

\begin{proof}[Alternate proof of Corollary~\ref{ucp_rt}]
    Since $I_m f(x, \xi) = 0$ for all $\xi \in \Sn$ and $x \in U$, the left hand side of \eqref{john_ray_relation} vanishes, and hence so does the right hand side. Since $f \in \Ec'(\Rb^n; S^m)$, $(Rf)_{i_1 j_1 \dots i_m j_m} \in \Ec'(\Rb^n)$ and thus by  \cite[Theorem 1.2] {Keijo_partial_function}, $Rf$ vanishes identically on $\Rb^n$. The proof now follows from \cite[Theorem 2.5.1]{Sharafutdinov_Book}.
\end{proof}

%\section{Momentum Ray Transform}\label{MRT}
\section{ucp for momentum ray transforms}\label{UCP:MRT}
In this section we study unique continuation for momentum ray transforms. We prove  an analogue of the identity proved for ray transforms (Proposition \ref{key_lemma_RT}), for momentum ray transforms in Proposition \ref{MRT:Prop} below. We first show the equivalence of $R^k$ and $W^k$ and then we prove a lemma required in the proof of Proposition \ref{key_lemma_RT}. % which will be used in Proposition \ref{MRT:Prop} below.

\begin{lemma}\label{rkwk_equivalence}
	Let $f \in \Dc'(\Rb^n; S^m)$ and $U \subseteq \Rb^n$ be open. Then for $0 \leq k \leq m$, $W^k f|_U = 0$ if and only if $R^k f|_U = 0$. In fact, the following equalities hold
	\begin{align}
		(W^k f)_{p_1 \dots p_{m-k} q_1 \dots q_{m-k} i_1 \dots i_k} &= 2^{m-k}\sigma(q_1 \dots q_{m-k} i_1 \dots i_k) \sigma(p_1 \dots p_{m-k}) \label{GRtoGW}
		\\ & \hspace{5cm} (R^k f)_{p_1 q_1 \dots p_{m-k} q_{m-k} i_1 \dots i_k} \notag
		\\ (R^k f)_{p_1 q_1 \dots p_{m-k} q_{m-k} i_1 \dots i_k} &= \frac{1}{m-k+1} \binom{m}{k} \alpha(p_1 q_1) \dots \alpha(p_{m-k} q_{m-k}) \label{GWtoGR}
		\\ & \hspace{5cm} (W^k f)_{p_1 \dots p_{m-k} q_1 \dots q_{m-k} i_1 \dots i_k}. \notag
	\end{align}
\end{lemma}

\begin{proof}
	%We will prove the theorem by establishing an explicit relation between $W^k f$ and $R^k f$.
	We present the  proof using similar ideas  from \cite[Lemma 2.4.2]{Sharafutdinov_Book}. %(Lemma $2.4.2$).
	From \cite[Equation 27]{SumanSIAM} we obtain 
	\begin{align*}
		(W^k f)_{p_1 \dots p_{m-k} q_1 \dots q_{m-k} i_1 \dots i_k} = \sigma(q_1 \dots q_{m-k} i_1 \dots i_k) (Wf^{i_1 \dots i_k})_{p_1 \dots p_{m-k} q_1 \dots q_{m-k}}.
	\end{align*}
	Using \eqref{RtoW} for an $m-k$-tensor field $f^{i_1 \dots i_k}$ we get \eqref{GRtoGW}.
	
	Now, we prove \eqref{GWtoGR}. Decomposing the symmetrizations $\sigma(p_1 \dots p_{m-k})$ and $\sigma(q_1 \dots q_{m-k} i_1 \dots i_k)$ in the definition of $W^k$ with respect to the indices $p_{m-k}$ and $q_{m-k}$ and taking into account the symmetries of $f$ and mixed partial derivatives, we get
	\begin{align*}
		(W^k f)_{p_1 \dots p_{m-k} q_1 \dots q_{m-k} i_1 \dots i_k} &= \sigma(p_1 \dots p_{m-k-1}) \sigma(q_1 \dots q_{m-k-1} i_1 \dots i_k) \sum\limits_{l=0}^{m-k} (-1)^l \binom{m-k}{l} \frac{1}{m(m-k)}
		\\ & \hspace{6mm} \times \bigg[ (m-k-l)^2 \frac{\p^{m-k} f^{i_1 \dots i_k}_{p_1 \dots p_{m-k-l-1} p_{m-k} q_1 \dots q_l}}{\p x_{p_{m-k-l}} \dots \p x_{p_{m-k-1}} \p x_{q_{l+1}} \dots \p x_{q_{m-k}}}
		\\ & \hspace{1cm} + l^2 \frac{\p^{m-k} f^{i_1 \dots i_k}_{p_1 \dots p_{m-k-l} q_1 \dots q_{l-1} q_{m-k}}}{\p x_{p_{m-k-l+1}} \dots \p x_{p_{m-k}} \p x_{q_{l}} \dots \p x_{q_{m-k-1}}} +h \bigg],
	\end{align*}
	where $h$ is some tensor symmetric in $p_{m-k}, q_{m-k}$. 
	Applying the operator $\alpha(p_{m-k} \,  q_{m-k})$ to this equality and noting that it commutes with $\sigma(p_1 \dots p_{m-k-1})$ and $\sigma(q_1 \dots q_{m-k-1} i_1 \dots i_k)$, we obtain
	\begin{align*}
		\alpha(p_{m-k} \, q_{m-k}) (W^k f)_{p_1 \dots p_{m-k} q_1 \dots q_{m-k} i_1 \dots i_k} &= \alpha(p_{m-k} \, q_{m-k}) \sigma(p_1 \dots p_{m-k-1}) \sigma(q_1 \dots q_{m-k-1} i_1 \dots i_k)
		\\ & \hspace{6mm} \sum\limits_{l=0}^{m-k} (-1)^l \binom{m-k}{l} \frac{1}{m(m-k)}
		\\ & \hspace{6mm} \times \bigg[ (m-k-l)^2 \frac{\p^{m-k} f^{i_1 \dots i_k}_{p_1 \dots p_{m-k-l-1} p_{m-k} q_1 \dots q_l}}{\p x_{p_{m-k-l}} \dots \p x_{p_{m-k-1}} \p x_{q_{l+1}} \dots \p x_{q_{m-k}}}
		\\ & \hspace{1cm} - l^2 \frac{\p^{m-k} f^{i_1 \dots i_k}_{p_1 \dots p_{m-k-l} p_{m-k} q_1 \dots q_{l-1}}}{\p x_{p_{m-k-l+1}} \dots \p x_{p_{m-k-1}} \p x_{q_{l}} \dots \p x_{q_{m-k}}}  \bigg],
	\end{align*}
	where we interchanged $p_{m-k}$ and $q_{m-k}$ in the last term, contributing to the negative sign. 
	
	Combining the first summand in the brackets of the $l$-th term of the sum with the second summand of the $(l+1)$-th term, we get
	\begin{align*}
		&  \alpha(p_{m-k} \, q_{m-k}) (W^k f)_{p_1 \dots p_{m-k} q_1 \dots q_{m-k} i_1 \dots i_k}\\ =& \frac{m-k+1}{m} \alpha(p_{m-k} \, q_{m-k}) \sigma(p_1 \dots p_{m-k-1}) \sigma(q_1 \dots q_{m-k-1} i_1 \dots i_k)
		\\& \sum\limits_{l=0}^{m-k-1} (-1)^l \binom{m-k-1}{l} \frac{\p^{m-k} f^{i_1 \dots i_k}_{p_1 \dots p_{m-k-l-1} p_{m-k}} q_1 \dots q_l}{\p x_{p_{m-k-l}} \dots \p x_{p_{m-k-1}} \p x_{q_{l+1}} \dots \p x_{q_{m-k}}}\\
		%\intertext{which can be interpreted as}
		&= \frac{m-k+1}{m} \alpha(p_{m-k} q_{m-k}) \frac{\p}{\p x_{q_{m-k}}} (W^k f^{p_{m-k}})_{p_1 \dots p_{m-k-1} q_1 \dots q_{m-k-1} i_1 \dots i_k},
	\end{align*}
	where on the right hand side $W^k$ acts on the $(m-1)$-tensor field obtained by fixing the index $p_{m-k}$ in $f$.
	Iterating this $(m-k)$-times, we get
	\begin{align}
		\alpha(p_1 q_1) \dots \alpha(p_{m-k} q_{m-k}) (W^k f)_{p_1 \dots p_{m-k} q_1 \dots q_{m-k} i_1 \dots i_k} 
		&= \frac{(m-k+1)}{\binom{m}{k}} \alpha(p_1 q_1) \dots \alpha(p_{m-k} q_{m-k})
		\\ & \hspace{1cm} \frac{\p^{m-k}}{\p x_{q_1} \dots \p x_{q_{m-k}}} (W^k f^{p_1 \dots p_{m-k}})_{i_1 \dots i_k}.   \notag
	\end{align}
	Finally, note that on the right hand side, $W^k$ acts on the $k$-tensor obtained by fixing $(m-k)$-indices. Since by our convention, $W^k$ acting on $k$-tensors is just the identity operator, \eqref{GWtoGR} follows.
\end{proof}
%\dbc{although the terms are out of margin, nothing is cut of paper.}

\begin{remark}
	From the expression for $R^k$ in \eqref{an_alternate_definition}, it is clear that given $R^k$ for some $0 \leq k \leq m$, we can recover $R^s$ for all $0 \leq s < k$. Since $W^k$ and $R^k$ are equivalent, we can recover $W^s$ for $0 \leq s \leq k$ as well.
\end{remark} 

\begin{lemma}\label{key_lemma_mrt1}
	Let $f \in C_c^{\infty}(\Rb^n; S^m)$. For $m \geq 0$ and $0 \leq k \leq m$, 
	\begin{equation}\label{key_equation_mrt1}
		\int\limits_{\mathbb{S}^{n-1}} \xi_{i_1} \dots \xi_{i_{m-k}} \, (J_m^kf)(x,\xi) \, \Dxi = \sum\limits_{r=0}^k (-1)^{k-r} \frac{1}{r!} \binom{k}{r} \lb j_{x^{\odot k-r}} \delta^r N_m^rf \rb_{i_1 \cdots i_{m-k}},
	\end{equation}
where the operator $j$ is given in \eqref{jop} and $\delta$ is given in \eqref{def_of_delta}.
\end{lemma}

%\subsection{Proof of main results} We begin by proving lemma~\ref{key_lemma_mrt1}.

\begin{proof}%[Proof of Lemma~\ref{key_lemma_mrt1}]
	We prove this result by induction on $k$. 
	For $k=0$, \eqref{key_equation_mrt1} is just the definition of the normal operator of the ray transform of a symmetric $m$-tensor field $f$ given in \eqref{Normal}. 
	
	Now assume that \eqref{key_equation_mrt1} is true for all $0,1, \dots, k-1$ and we want to prove this for $k$.
	By \eqref{div_nmrt}, we have,
	\begin{align*}
		\int\limits_{\mathbb{S}^{n-1}} \xi_{i_1} \dots \xi_{i_{m-k}}  \, (J_m^kf)(x,\xi) \, \Dxi&= \frac{1}{k!} (\delta^k N_m^k f)_{i_1 \dots i_{m-k}} (x)  
		\\ & -\sum\limits_{l=0}^{k-1} \binom{k}{l} \int\limits_{\mathbb{S}^{n-1}} \langle x, \xi\rangle^{k-l} \xi_{i_1} \dots \xi_{i_{m-k}} (J_m^l f)(x, \xi) \, \Dxi.
	\end{align*}
	
	Since $\langle x,\xi\rangle^{r}= j_{x^{\odot r}}\xi^{\odot r}$, we have,
	%Observe that $(x \cdot \xi)^r$ can be written as $\langle x^r, \xi^r \rangle$. This
	together with the induction hypothesis,
	\begin{align}\label{eq_4.4}
		\int\limits_{\mathbb{S}^{n-1}} & \xi_{i_1} \dots \xi_{i_{m-k}}  \, (J_m^kf)(x,\xi) \, \Dxi \nonumber \\ &= \frac{1}{k!} (\delta^k N_m^k f)_{i_1 \dots i_{m-k}} (x)
		 - \sum\limits_{l=0}^{k-1} \binom{k}{l} j_{x^{\odot k-l}} \lb \int\limits_{\mathbb{S}^{n-1}} \xi_{i_1} \dots \xi_{i_{m-k}} \xi^{\odot k-l} (J_m^lf)(x,\xi) \Dxi  \rb
		\nonumber  \\ \nonumber&= \frac{1}{k!} (\delta^k N_m^k f)_{i_1 \dots i_{m-k}} (x)
		- \sum\limits_{l=0}^{k-1} \binom{k}{l}  \sum\limits_{r=0}^l (-1)^{l-r} \frac{1}{r!} \binom{l}{r} \lb j_{x^{\odot k-r}} \delta^r N_m^rf \rb_{i_1 \dots i_{m-k}}
		\nonumber   \\ 
		\intertext{Interchanging the order of summation in the second term,}
		&= \frac{1}{k!} (\delta^k N_m^k f)_{i_1 \dots i_{m-k}} (x)-\sum\limits_{r=0}^{k-1} \frac{1}{r!} \lb j_{x^{\odot k-r}} \delta^r N_m^rf \rb_{i_1 \cdots i_{m-k}} 
		\sum\limits_{l=r}^{k-1} \binom{k}{l} (-1)^{l-r} \binom{l}{r}.
	\end{align}
	
	% \begin{align*}
	%   \tred{ \sum x_{j_1}\cdots x_{j_{k-l}} \xi_{j_1}\cdots \xi_{j_{k-l}} (\xi_{i_1} \dots \xi_{i_{m-k}}) }
	% \end{align*}

	We note that   
	\begin{align}\label{eq_4.5}
		\sum\limits_{l=r}^k (-1)^{l-r} \binom{k}{l} \binom{l}{r} = 0.
	\end{align}
	This is obtained by differentiating the binomial expansion of $(1+x)^k$, $r$ number of times and letting $x=-1$. From \eqref{eq_4.5}, we have $\sum\limits_{l=r}^{k-1} \binom{k}{l} (-1)^{l-r} \binom{l}{r} = (-1)^{k-r+1}\binom{k}{r}$.
	Combining this with \eqref{eq_4.4} we get
	\begin{align*}
		\int\limits_{\mathbb{S}^{n-1}} \xi_{i_1} \dots \xi_{i_{m-k}}  \, (J_m^kf)(x,\xi) \, \Dxi &= \sum\limits_{r=0}^k (-1)^{k-r} \frac{1}{r!} \binom{k}{r} \lb j_{x^{\odot k-r}} \delta^r N_m^rf \rb_{i_1 \cdots i_{m-k}}.
	\end{align*}
	This completes the proof.
\end{proof}

\begin{proposition}\label{MRT:Prop}
	 For $f \in \Ec'(\Rb^n; S^m)$,
\begin{equation}
    \begin{split}\label{key_equation_mrt}
		m! \, & N_0 ((Rf^{i_1 \dots i_k})_{p_1 q_1 \dots p_{m-k} q_{m-k}})   \\&=  \sigma(i_1 \dots i_k) \,\sum\limits_{r=0}^k (-1)^r \binom{k}{r} \frac{\partial^{r}}{\partial x_{i_1} \dots \partial x_{i_r} }  (R^k(G_{m-r}))_{p_1 q_1 \dots p_{m-k} q_{m-k} i_{r+1} \dots i_k},
    \end{split}
\end{equation}
where $G_{m-r}$ is a symmetric $(m-r)$-tensor given as

$ \lb G_{m-r}\rb_{i_{r+1} \dots i_k q_1 \dots q_{m-k}} = \sum\limits_{l=0}^{\lfloor\frac{m-r}{2} \rfloor} c_{l, m-r} \, i^l j^l \Bigg( \sum\limits_{p=0}^r (-1)^{r-p} \frac{1}{p!} \binom{r}{p}  (j_{x^{\odot r-p}} \delta^p N_m^p f)_{i_{r+1} \dots i_k q_1 \dots q_{m-k}}  \Bigg).$
\end{proposition}
% \begin{remark}
% This proposition immediately given the uniqueness of mrt based on normal operator given in $\mathbb{R}^n$.
% \end{remark}

\bpr As in the proof of Proposition \ref{key_lemma_RT}, we prove the identity for $f\in C_c^{\infty}(\Rb^n; S^m)$.

 We begin with the following relation which can be shown in exactly the same way as in \eqref{john_ray_relation} by considering the ray transform of the $m-k$ symmetric tensor field obtained by fixing the indices $i_1, \cdots, i_k$ and then applying the John operator $m-k$ times %see also \cite{....}. %   \tbl{!!add ref!!}%\footnote{add reference (Generalized Saint-Venant operator and integral moment transforms) by R.K.Mishra and S.K.Sahoo}(see equation $16$):
\begin{equation}
	(m-k)! (-2)^{m-k} I_0((R f^{i_1 \dots i_k})_{p_1 q_1 \dots p_{m-k} q_{m-k}}) =  J_{p_1 q_1} \cdots J_{p_{m-k} q_{m-k}}I_{m-k} \lb f^{i_1 \dots i_k}\rb.
\end{equation}
%where on the LHS, $J^{m-k}$ denotes the John operator acting on the ray transform of $m-k$ tensor $f^{i_1 \dots i_k}$ keeping the indices $i_1 \dots i_k$ fixed, and the RHS denotes the ray transform of the scalar function $(R f^{i_1 \dots i_k})_{p_1 q_1 \dots p_{m-k} q_{m-k}}$. 

Substituting the expression for $J_{m-k}^0 f^{i_1 \dots i_k}$ from \cite[Lemma 4.2]{SumanSIAM} %in terms of derivatives of momentum ray transform of $f$,  
we get
\begin{align*}
	(-2)^{m-k}(m-k)! \,& J_0^0((Rf^{i_1 \dots i_k})_{p_1 q_1 \dots p_{m-k} q_{m-k}}) \\ &= 2^{m-k} \alpha(p_1 q_1) \dots \alpha(p_{m-k} q_{m-k}) \frac{\partial^{2m-2k}}{\partial x_{p_1} \dots \partial x_{p_{m-k}} \partial \xi_{q_1} \dots \partial \xi_{q_{m-k}}} 
	\\& \,\, \,\,\, \left [ \frac{(m-k)!}{m!} \sigma(i_1 \dots i_k) \sum\limits_{r=0}^k (-1)^r \binom{k}{r} \frac{\partial^k J_m^rf}{\partial x_{i_1} \dots \partial x_{i_r} \partial \xi_{i_{r+1}} \dots \partial \xi_{i_k}}  \right].
\end{align*}
Integrating both sides over $\mathbb{S}^{n-1}$, we get
\begin{align*}
	(-1)^{m-k}(m-k)! \, & N_0((Rf^{i_1 \dots i_k})_{p_1 q_1 \dots p_{m-k} q_{m-k}})  \\ &= \alpha(p_1 q_1) \dots \alpha(p_{m-k} q_{m-k}) \frac{(m-k)!}{m!} \sigma(i_1 \dots i_k)
	\\ &\quad  \Bigg [ \sum\limits_{r=0}^k (-1)^r \binom{k}{r} \frac{\partial^{m-k+r}}{\partial x_{i_1} \dots \partial x_{i_r} \partial x_{p_1} \dots \partial x_{p_{m-k}}}
	\\ & \qquad  \bigg \{ \int\limits_{\mathbb{S}^{n-1}} \frac{\partial^{m-r} J_m^rf}{\partial \xi_{i_{r+1}} \dots \partial \xi_{i_k} \partial \xi_{q_1} \dots \partial \xi_{q_{m-k}}} \, \Dxi \bigg\} \Bigg].
\end{align*}

Since $J_m^r f$ is a homogeneous function of degree $m-r-1$ in the $\xi$ variable, using Lemma \ref{IBPlemma}, we have
\begin{align*}
	(-1)^{m-k} (m-k)! \, & N_0^0 ((Rf^{i_1 \dots i_k})_{p_1 q_1 \dots p_{m-k} q_{m-k}})  \\ &= \alpha(p_1 q_1) \dots \alpha(p_{m-k} q_{m-k}) \frac{(m-k)!}{m!} \sigma(i_1 \dots i_k)\times
	\\ &\quad  \Bigg [ \sum\limits_{r=0}^k (-1)^r \binom{k}{r} \frac{\partial^{m-k+r}}{\partial x_{i_1} \dots \partial x_{i_r} \partial x_{p_1} \dots \partial x_{p_{m-k}}}
	\\ & \qquad \bigg \{ \sum\limits_{l=0}^{\lfloor\frac{m-r}{2} \rfloor} c_{l, m-r} \, i^l j^l \int\limits_{\Sn} \xi_{i_{r+1}} \dots \xi_{i_k} \xi_{q_1} \dots \xi_{q_{m-k}}  J_m^r f \, \Dxi   \bigg\} \Bigg],
\end{align*}
where the constants $c_{l, m-r}$ are  given in Lemma~\ref{IBPlemma}.

Finally, we use Lemma~\ref{key_lemma_mrt1} to express the last integral in terms of the normal operator,
\begin{align*}
	(m-k)! \,(-1)^{m-k} & N_0^0 ((Rf^{i_1 \dots i_k})_{p_1 q_1 \dots p_{m-k} q_{m-k}})  \\ &= \alpha(p_1 q_1) \dots \alpha(p_{m-k} q_{m-k}) \frac{(m-k)!}{m!} \sigma(i_1 \dots i_k)\times
	\\ &\quad  \Bigg [ \sum\limits_{r=0}^k (-1)^r \binom{k}{r} \frac{\partial^{m-k+r}}{\partial x_{i_1} \dots \partial x_{i_r} \partial x_{p_1} \dots \partial x_{p_{m-k}}}
	\\ & \qquad \bigg \{ \sum\limits_{l=0}^{\lfloor\frac{m-r}{2} \rfloor} c_{l, m-r} \, i^l j^l \big( \sum\limits_{p=0}^r (-1)^{r-p} \frac{1}{p!} \binom{r}{p}  (j_{x^{\odot r-p}} \delta^p N_m^p f)_{i_{r+1} \dots i_k q_1 \dots q_{m-k}}  \big) \bigg\} \Bigg]
	\\&=  \alpha(p_1 q_1) \dots \alpha(p_{m-k} q_{m-k}) \frac{(m-k)!}{m!} \sigma(i_1 \dots i_k)\times
	\\ &\quad  \Bigg [ \sum\limits_{r=0}^k (-1)^r \binom{k}{r} \frac{\partial^{m-k+r}}{\partial x_{i_1} \dots \partial x_{i_r} \partial x_{p_1} \dots \partial x_{p_{m-k}}} (G_{m-r})_{q_1 \dots q_{m-k} i_{r+1} \dots i_k} \Bigg].
\end{align*}
Since $\alpha(p_j q_j)$ for $1 \leq j \leq m-k$ commutes with $\sigma(i_1 \dots i_k)$, we take $\alpha(p_j q_j)$ inside the summation, and use the anti-symmetry of $R^k$ to finish the proof. 
%This implies
%\begin{align*}
%	(m-k)! \, & N_0^0 ((Rf^{i_1 \dots i_k})_{p_1 q_1 \dots p_{m-k} q_{m-k}})  \\ &= \frac{(m-k)!}{m!} \sigma(i_1 \dots i_k)\,R^k\,\sum\limits_{r=0}^k (-1)^r \binom{k}{r} \frac{\partial^{r}}{\partial x_{i_1} \dots \partial x_{i_r} }  (G_{m-r}).
%\end{align*}
\epr
With these preliminaries, let us prove Theorems \ref{mrt_ucp} and \ref{mrt_ucp_1}.
\begin{proof}[Proof of Theorem \ref{mrt_ucp}]
Since $N^p_m f|_{U}=0$ for $0\leq p\leq k$, the right hand side  of \eqref{key_equation_mrt} vanishes in $U$. This implies 
\begin{align*}
     N_0 ((Rf^{i_1 \dots i_k})_{p_1 q_1 \dots p_{m-k} q_{m-k}})=0 \quad \mbox{in} \quad U.
\end{align*}
By the definition in \eqref{an_alternate_definition}, $ R^kf|_{U}=0$ implies 
\begin{align*}
(Rf^{i_1 \dots i_k})_{p_1 q_1 \dots p_{m-k} q_{m-k}}=0   \quad \mbox{in} \quad U. 
\end{align*}
Applying unique continuation for the normal operator of the ray transform of scalar functions  \cite[Theorem 1.1]{Keijo_partial_function} on $(Rf^{i_1 \dots i_k})_{p_1 q_1 \dots p_{m-k} q_{m-k}}$, we conclude that  \[ (Rf^{i_1 \dots i_k})_{p_1 q_1 \dots p_{m-k} q_{m-k}}=0 \quad \mbox{in} \quad \mathbb{R}^n, \] for all $1 \leq i_1, \dots ,i_k, p_1, \dots, p_{m-k}, q_1, \dots, q_{m-k} \leq n$. Again using \eqref{an_alternate_definition}, we get $R^k f=0$  in $\Rb^n$. Combining Lemma \ref{rkwk_equivalence} and \cite[Theorem 2.17.2]{Sharafutdinov_Book}, we conclude the proof.
\end{proof}

\bpr[Proof of Theorem~\ref{mrt_ucp_1}]

By \cite[Lemma 4.8]{BKS}, we have that $J_m^k f$ determines $J_m^{r}f$ for all $r<k$. Hence $I_m^{0}f,\cdots, I_m^{k}f$ are determined on $\lb U\times \Sb^{n-1} \rb \cap T\Sb^{n-1}$. This then implies that we know $N_m^p f|_U$ for all $0\leq p\leq k$. Now using Theorem \ref{mrt_ucp}, we 
have the result.
\epr

\section{ucp for transverse ray transform}\label{UCP:TRT}
\bpr[Proof of Theorem \ref{ucp_trt}]
We proceed by an argument similar to the one used in \cite{Sharafutdinov_Book}. Let $f$ be a compactly supported tensor field distribution. Fix a non-zero vector $\eta\perp \xi$, and consider the compactly supported distribution: 
\[
\vp_{\eta}(x)= f_{i_1\cdots i_m}(x)\eta_{i_1}\cdots\eta_{i_m}.
\] % Its transverse ray transform (TRT) \cite[Chapter 5]{Sharafutdinov_Book} is defined as: 
The ray transform of $\vp_{\eta}$ is well-defined. 
%\begin{equation}\label{def_transverse}
   % Tf(x,\xi,\eta)= \int\limits_{-\infty}^{\infty} f_{i_1\cdots i_m}(x+t\xi)\, \eta_{i_1}\cdots\eta_{i_m}\, dt,\quad \eta \perp \xi.
%\end{equation}
Fix $x\in \mbox{supp}f$. Denote $V_H=V\cap H_{\eta} $, where $H_{\eta}$ is a hyperplane with normal $\eta$ and passing through $x$.  Note that $ V_{H}$ is an open set in $ \mathbb{R}^{n-1}$. For $ x\in V_{H}$, define $ \varphi_{\eta} (x)= f_{i_1\cdots i_m}(x)\eta_{i_1}\cdots\eta_{i_m}$ for a fixed $\eta$. From the knowledge of the transverse ray transform $\Tc f$, we have that 
\begin{align*}
    I^0 (\phi_{\eta})=0 \quad \mbox{and} \quad \phi_{\eta}=0 \quad \mbox{in} \quad V_H.
\end{align*}
Using unique continuation for scalar functions by \cite{Keijo_partial_function} we get $ \phi_{\eta}=0 $ in $ H_{\eta}$. We can vary $ \eta$ in an open cone $\mathcal{C}$ (say) and obtain $ \phi_{\eta} =0 $ for all $ \eta \in \mathcal{C}$. Any such cone always contains $ n$ linearly independent vectors  say $ \eta_1,\cdots, \eta_n$. Then the collection of  $\binom{m+n-1}{m}$ symmetric tensors  \[ A=\{\eta_{i_1}\odot \eta_{i_2}\odot \cdots \odot \eta_{i_m}: 1\le i_1,\cdots i_m\le n \}\] are linearly independent. This can be proved directly; see also \cite[Lemma 5.4]{Venky_and_Rohit}. This gives
\begin{align*}
     \langle f,\eta^{\odot m}\rangle =0 \quad \mbox{for all} \quad \eta \in A,
\end{align*}
which in turn gives $ f(x)=0$ for fixed $x$. Varying $x\in \mbox{supp}f$ we get $ f\equiv 0$.
% From \eqref{def_transverse} we have that ray transform of $\varphi_{\eta}$ vanishes on $V_H=V\cap H_{\eta}$, where . Since $V_H$ is an open set therefore by \\ref{}, we have $\varphi_{\eta} (x)=0$ for all $x\in V$. Now we choose $n $ linearly independent $\eta$ say $\{\eta_1,\eta_2,\cdots,\eta_n\}$ in $\mathbb{R}^n,$ $n\ge 3$ such that $\{\eta_i \otimes \eta_j\}$ is a basis of $\mathbb{R}^{\frac{n(n+1)}{2}}$. Having this information is enough to conclude that $ f=0$ in $V$.
\epr

\section*{Acknowledgments}
%We express our gratitude to Prof.\ Vladimir A.\ Sharafutdinov for his inspiring monograph, ``Integral Geometry of Tensor Fields" which has helped us immensely over the years and without which the results of this paper would not have been possible. 
S.K.S. was supported by Academy of Finland (Centre of Excellence in Inverse Modelling and Imaging, grant  284715) and European Research Council under Horizon 2020 (ERC CoG 770924).

\bibliography{references.bib}

\def\dbar{\leavevmode\hbox to 0pt{\hskip.2ex \accent"16\hss}d}
\begin{thebibliography}{10}

\bibitem{Abhishek-Mishra}
Anuj Abhishek and Rohit~Kumar Mishra.
\newblock Support theorems and an injectivity result for integral moments of a
  symmetric {$m$}-tensor field.
\newblock {\em J. Fourier Anal. Appl.}, 25(4):1487--1512, 2019.

\bibitem{BKS}
Sombuddha Bhattacharyya, Venkateswaran~P. Krishnan, and Suman~Kumar Sahoo.
\newblock Unique determination of anisotropic perturbations of a polyharmonic
  operator from partial boundary data, 2021.

\bibitem{Denisjuk_Paper}
Alexander Denisjuk.
\newblock Inversion of the x-ray transform for 3{D} symmetric tensor fields
  with sources on a curve.
\newblock {\em Inverse Problems}, 22(2):399--411, 2006.

\bibitem{Desai-Lionheart}
Naeem~M. Desai and William R.~B. Lionheart.
\newblock An explicit reconstruction algorithm for the transverse ray transform
  of a second rank tensor field from three axis data.
\newblock {\em Inverse Problems}, 32(11):115009, 19, 2016.

\bibitem{GSU-fractional}
Tuhin Ghosh, Mikko Salo, and Gunther Uhlmann.
\newblock The {C}alder\'{o}n problem for the fractional {S}chr\"{o}dinger
  equation.
\newblock {\em Anal. PDE}, 13(2):455--475, 2020.

\bibitem{Keijo_partial_function}
Joonas Ilmavirta and Keijo M\"{o}nkk\"{o}nen.
\newblock Unique continuation of the normal operator of the x-ray transform and
  applications in geophysics.
\newblock {\em Inverse Problems}, 36(4):045014, 23, 2020.

\bibitem{Keijo_partial_vector_field}
Joonas Ilmavirta and Keijo M\"{o}nkk\"{o}nen.
\newblock X-ray tomography of one-forms with partial data.
\newblock {\em SIAM J. Math. Anal.}, 53(3):3002--3015, 2021.

\bibitem{John_diff_paper}
Fritz John.
\newblock The ultrahyperbolic differential equation with four independent
  variables.
\newblock {\em Duke Math. J.}, 4(2):300--322, 1938.

\bibitem{Kotake-Narasimhan}
Takeshi Kotake and Mudumbai~S. Narasimhan.
\newblock Regularity theorems for fractional powers of a linear elliptic
  operator.
\newblock {\em Bull. Soc. Math. France}, 90:449--471, 1962.

\bibitem{Krishnan2018}
Venkateswaran~P. Krishnan, Ramesh Manna, Suman~Kumar Sahoo, and Vladimir~A.
  Sharafutdinov.
\newblock Momentum ray transforms.
\newblock {\em Inverse Probl. Imaging}, 13(3):679--701, 2019.

\bibitem{Krishnan2019a}
Venkateswaran~P. Krishnan, Ramesh Manna, Suman~Kumar Sahoo, and Vladimir~A.
  Sharafutdinov.
\newblock Momentum ray transforms, {II}: range characterization in the
  {S}chwartz space.
\newblock {\em Inverse Problems}, 36(4):045009, 33, 2020.

\bibitem{Venky_and_Rohit}
Venkateswaran~P. Krishnan and Rohit~Kumar Mishra.
\newblock Microlocal analysis of a restricted ray transform on symmetric
  {$m$}-tensor fields in {$\Bbb{R}^n$}.
\newblock {\em SIAM J. Math. Anal.}, 50(6):6230--6254, 2018.

\bibitem{krishnan2021ray}
Venkateswaran~P. Krishnan and Vladimir~A. Sharafutdinov.
\newblock Ray transform on sobolev spaces of symmetric tensor fields, i: Higher
  order reshetnyak formulas, 2021.

\bibitem{Lionheart-Withers}
W.~R.~B. Lionheart and P.~J. Withers.
\newblock Diffraction tomography of strain.
\newblock {\em Inverse Problems}, 31(4):045005, 17, 2015.

\bibitem{Lionheart-Sharafutdinov}
William Lionheart and Vladimir Sharafutdinov.
\newblock Reconstruction algorithm for the linearized polarization tomography
  problem with incomplete data.
\newblock In {\em Imaging microstructures}, volume 494 of {\em Contemp. Math.},
  pages 137--159. Amer. Math. Soc., Providence, RI, 2009.

\bibitem{SumanSIAM}
Rohit~Kumar Mishra and Suman~Kumar Sahoo.
\newblock Injectivity and range description of integral moment transforms over
  {$m$}-tensor fields in {$\Bbb{R}^n$}.
\newblock {\em SIAM J. Math. Anal.}, 53(1):253--278, 2021.

\bibitem{Novikov-Sharafutdinov}
Roman Novikov and Vladimir Sharafutdinov.
\newblock On the problem of polarization tomography. {I}.
\newblock {\em Inverse Problems}, 23(3):1229--1257, 2007.

\bibitem{Riesz}
Marcel Riesz.
\newblock Int\'{e}grales de riemann-liouville et potentiels.
\newblock {\em Acta Sci. Math. (Szeged)}, 9(1-1):1--42, 1938-40.

\bibitem{Ruland_fractional}
Angkana R\"{u}land.
\newblock Unique continuation for fractional {S}chr\"{o}dinger equations with
  rough potentials.
\newblock {\em Comm. Partial Differential Equations}, 40(1):77--114, 2015.

\bibitem{Sharafutdinov_Book}
V.~A. Sharafutdinov.
\newblock {\em Integral geometry of tensor fields}.
\newblock Inverse and Ill-posed Problems Series. VSP, Utrecht, 1994.

\bibitem{Weyl_lemma}
Daniel~W. Stroock.
\newblock Weyl's lemma, one of many.
\newblock In {\em Groups and analysis}, volume 354 of {\em London Math. Soc.
  Lecture Note Ser.}, pages 164--173. Cambridge Univ. Press, Cambridge, 2008.

\bibitem{TrevesBook}
Fran\c{c}ois Tr\`eves.
\newblock {\em Topological vector spaces, distributions and kernels}.
\newblock Academic Press, New York-London, 1967.

\bibitem{Uhlmann-TTT}
Gunther Uhlmann.
\newblock Travel time tomography.
\newblock volume~38, pages 711--722. 2001.
\newblock Mathematics in the new millennium (Seoul, 2000).

\end{thebibliography}
\bibliographystyle{plain}

\end{document}